\documentclass[12pt, twoside]{article}
\usepackage[english]{babel}
\usepackage[T1]{fontenc}
\usepackage[utf8]{inputenc}
\usepackage{lmodern}
\usepackage[top=2.5cm,bottom=2.5cm,right=2.5cm,left=2cm]{geometry}
 \usepackage{amsmath, amssymb} 
 \usepackage{amsthm}
 \usepackage{graphicx}
 \usepackage{color}

 \usepackage{lmodern}

\usepackage{etoolbox}

\usepackage{hyperref}

\hypersetup{	
     colorlinks=true,       
      breaklinks=true,         
      urlcolor= magenta,       
      linkcolor=blue ,	
      citecolor=blue,	
                       }

\slshape

\newtheorem{theorem}{Theorem}[section]
\newtheorem{lemma}{Lemma}[section]
\newtheorem{corollary}{Corollary}[section]
\newtheorem{proposition}{Proposition}[section]
\newtheorem{remark}{Remark}[section]

\newtheorem*{lemma A1}{Lemma A1}
\newtheorem*{lemma A2}{Lemma A2}
\newtheorem*{k lemma}{Key Lemma}
\newcommand{\C}{\mathbb C}
\newcommand{\R}{\mathbb R}
\newcommand{\Z}{\mathbb Z}

\newcommand{\N}{\mathbb N}
\numberwithin{equation}{section}

\begin{document}

\title{\(L^2\)-Poisson integral representations of eigensections of invariant differential operators on a homogeneous line bundle over the complex Grassmann manifold \(SU(r,r+b)/S( U(r)\times U(r+b))\).}

\author{
Abdelhamid Boussejra \thanks{e-mail: boussejra.abdelhamid@uit.ac.ma}
Noureddine Imesmad \thanks{e-mail:noureddine.imesmad@uit.ac.ma}
Achraf Ouald Chaib\thanks{e-mail:achraf.oualdchaib@uit.ac.ma} \\
\begin{small}
Department of Mathematics,  Faculty of Sciences
\end{small}\\
\begin{small}
University Ibn Tofail, Kenitra, Morocco
\end{small}}

\maketitle

\begin{abstract}
 Let  \(E_l=G\times_K{\C}\) be the associated  homogeneous line bundle to a one dimensional \(K\)-representation  \(\tau_l\) (\(l\in {\Z}\)) over the noncompact complex Grassman manifold \(G/K\);  \(G=SU(r,r+b)\) and \(K=S(U(r)\times U(r+b))\). Let \(\mathbb{D}(E_l)\) be the algebra of \(G\)-invariant differential operators on \(E_l\). Let \(\lambda\) be a real and regular spectral parameter   in \(\mathfrak{a}^\ast\),  and let \(F\) be a solution of the system differential equations on \(E_l\): \(DF=\chi_{\lambda,l}(D)F\) for all \(D\) in \(\mathbb{D}(E_l)\). In this article we obtain  a necessary and sufficient condition for this \(F\) to be represented by the Poisson transform of \(f\) in the  section space \(L^2(K\times_M {\C})\). 
\end{abstract}


\section{Introduction and main result}
Let \(X=G/K\) be a Hermitian symmetric space. Fix an Iwasawa decomposition \(G=KAN\) and let \(P=MAN\) be the usual minimal parabolic subgroup associated to it.\\
Let \(\tau_l\) (\(l\in {\Z}\)) be a character of \(K\) and \(G\times_K{\C}\) the associated homogeneous line bundle on \(G/K\). As usual we identify its section space \(\Gamma^\infty(G\times_K{\C})\) with the space \(C^\infty(G,\tau_l)\) of smooth \({\C}\)-valued functions \(F\) on \(G\) satisfying  
\begin{align}\label{cov}
F(gk)=\tau_l(k)^{-1}F(g)\quad \textit{for all}\quad  g\in G, k\in K.
\end{align}
The Poisson transform \(\displaystyle P_{\lambda,l}\) for spectral parameter \(\displaystyle\lambda\in \mathfrak{a}^\ast_c\) ( the complex dual of \(\mathfrak{a}\)) is the linear map \(\displaystyle P_{\lambda,l}:\mathcal{B}(K\times_M {\C})\rightarrow C^\infty(G,\tau_l)\) from the space of hyperfunction valued sections of the homogeneous line bundle \(K\times_M {\C}\) to \(C^\infty(G,\tau_l)\), given by
\begin{align*}
P_{\lambda,l}f(g)=\int_K {\rm e}^{-(i\lambda+\rho)H(g^{-1}k)}\tau_l(\kappa(g^{-1}k))f(k)\,{\rm d}k,
\end{align*}
where \(\kappa\) and \(H\) are the Iwasawa projections  of \(G\) in \(K\) and \(\mathfrak{a}\) respectively, (\(\mathfrak{a}\) being the Lie algebra of \(A\)).\\
Let \(\mathbb{D}(E_l)\) be  the algebra of left-invariant differential operators on \(\displaystyle C^\infty(G,\tau_l)\) (which is commutative). In \cite{Sh1} Shimeno showed that the Poisson transform  
\(P_{\lambda,l}\) is an isomorphism from \(\mathcal{B}(K\times_M {\C})\) onto \(\mathcal{E}_{\lambda,l}(G)\) the solution space of the system of differential equations
\begin{align*}
DF=\chi_{\lambda,l}(D)F, \quad D\in \mathbb{D}(E_l),
\end{align*}
under certain conditions on the parameter $\lambda$, see section 2 below for more details.\\
In the light of this result, it is natural to look for a characterization of those \(F\) in \(\mathcal{E}_{\lambda,l}(G)\) for which the boundary value \(f\) is an \(L^2\)-section of the homogeneous line bundle \(K\times_M {\C}\) for real and spectral parameter \(\lambda\in \mathfrak{a}_{reg}^{*}\). This amount to characterize the image of the Poisson transform \(P_{\lambda,l}\) of \(L^2(K\times_M {\C})\) for \(\lambda\in \mathfrak{a}_{reg}^{*}\).\\
Before stating the main result of this paper, we should mention to the reader that in the case \(\tau_l=1\), the problem of characterizing the \(L^2\)-range of the Poisson transform goes back to Strichartz \cite{S} who proved that \(P_\lambda\) is a topological isomorphism from \(L^2(K/M)\) onto a weighted \(L^2\)-  eigenspace of the Laplacian  in the case \(X\) is the real hyperbolic space. In the same paper he conjectured that this result can be extended to Riemannian symmetric spaces of higher rank ( cf. [\cite{S}, Conjecture 4.5], for more details).\\
An image characterization of the Poisson transform of \(L^2\)-functions on the boundary was obtained in \cite{BI} for the complex hyperbolic space and, for all  Riemannian symmetric spaces of rank one  by Ionescu in \cite{I}  (see also \cite{BI1},\cite{BS}, \cite{BO}, \cite{K}, \cite{KR}).
Finally in  \cite{Ka} Kaizuka settled the Strichartz conjecture for general Riemannian symmetric spaces.\\ 
Our main concern in this article is to extend Kaizuka's result  to homogeneous line bundles over $G/K$ when \(G=SU(r,r+b)\) and \(K=S(U(r)\times U(r+b))\).\\
More precisely, we identify the section space \(L^2(K\times_M{\C})\) with the space  \(L^2(K,\tau_l)\) consisting of all \({\C}\)-valued measurable (classes) functions \(f\) on \(K\)  satisfying \(f(km)=\tau_l(m)^{-1}f(k)\), for all \(k\in K, m\in M\), with
\begin{align*}
\parallel f\parallel_2=\left(\int_K\mid f(k\mid^2\, {\rm d}k\right)^{\frac{1}{2}}<\infty.
\end{align*}
Let \(L^2(G,\tau_l)\) be the space of all functions \(F:G\rightarrow {\C}\) that are \(L^2\) with respect to the Haar measure on \(G\) and satisfy the identity (\ref{cov}). It is obvious (from the unitarity of \(\tau_l\))  that \(\mid F(g)\mid\) is constant on cosets, so 
\begin{align}
\parallel F\parallel^2=\int_{G} \mid F(g)\mid^2\, {\rm d}g=\int_{G/K} \mid F(g)\mid^2\, {\rm d}g_K.
\end{align}
For \(\lambda\in \mathfrak{a}_{reg}^{*}\) we define \(\mathcal{E}^2_{\lambda,l}(G)\) to be the space of all functions \(F\in \mathcal{E}_{\lambda,l}(G)\) that satisfy
\[
\left\|F\right\|_*=\sup_{R>1}\dfrac{1}{R^{r/2}}\left(\int_{B(R)}|F(g)|^{2}dg_K\right)^\frac{1}{2}<\infty.
\]
Here \(B(R)\) is the open ball of radius \(R\) about the origin \(0=eK\),  \(d\) being the distance function on \(G/K\).\\
The main result we prove in this paper is the following:
\begin{theorem}\label{main R}
Let \(l\in {\Z}\) and \(\lambda\in\mathfrak{a}^\ast_\textit{reg}\).
\item[(i)]There exists a positive constant \(C>0\) depending only on \(l\) such that for every\\ \(f\) in \(L^2(K,\tau_l)
\) we have  following estimates
\begin{equation}\label{estimate0}
C^{-1}\mid c(\lambda,l)\mid \left\|f\right\|_2\leq \left\|P_{\lambda,l}f\right\|_*\leq C\mid  c(\lambda,l)\mid \left\|f\right\|_2,
\end{equation}
Furthermore we have
\begin{equation}\label{estimate00}
\lim_{R\rightarrow +\infty}\dfrac{1}{R^{r}}\int_{B(R)}|P_{\lambda,l}f(g)|^{2}{\rm d}g_K=2^{-r/2}\Gamma(r/2+1)^{-1}\mid c(\lambda,l)\mid^2\left\|f\right\|_2^2,
\end{equation}
where \(c(\lambda,l)\) is the Harish-Chandra \(c\)-function associated with \(\tau_l\) and \(\Gamma\) the usual Gamma function.
\item[(ii)] \(P_{\lambda,l}\) is a topological isomorphism from \(L^2(K,\tau_l)\) onto \(\mathcal{E}^2_{\lambda,l}(G)\).
\item[(iii)] The \(L^2\)-boundary \(f\) of the eigenfunction \(F\in \mathcal{E}^2_{\lambda,l}(G)\) is given by
\[
f(k)=2^{r/2}\Gamma(r/2+1)\mid c(\lambda,l)\mid^{-2}\lim_{R\rightarrow+\infty}
\frac{1}{R^r}\int_{B(R)}{\rm e}^{(i\lambda-\rho)H(g^{-1}k)} \tau_{-l}(\kappa(g^{-1}k))F(g){\rm d}g_K
\]
in \(L^2(K,\tau_l)\).
\end{theorem}
This  extends  Theorem 3.3 of Kaizuka \cite{Ka}, to the homogeneous line bundles over the bounded symmetric domains \(SU(r,r+b)/S(U(r)\times U(r+b))\).\\
\textbf{Consequence.}\\
For any \(\lambda\in \mathfrak{a}^\ast\) we define  the space
\(\mathcal{H}^2_{\lambda,l}(G)\) consisting of all \(F\in \mathcal{E}_{\lambda,l}(G)\) such that
 \[M_2(F)=\left(\limsup_{R\rightarrow \infty}\frac{1}{R^r}\int_{B(R)}\mid F(g)\mid^2\, {\rm d}\,g_K\right)^{\frac{1}{2}}.
\]
Then as an immediate consequence of Theorem \ref{main R} we get the following result which is closely similar to O. Bray conjecture in the case of the real hyperbolic space (see \cite{Br}).
\begin{corollary}\label{cor2}
If \(\lambda\in\mathfrak{a}^\ast_\textit{reg}\) then \(\mathcal{H}^2_{\lambda,l}(G)\) is a Banach space.
\end{corollary}
\begin{remark}
\begin{itemize}
\item[(i)] For \(l=0\) and \(r=1\) the above corollary  was proved by Ionescu \cite{I}.
\item[(ii)] For \(l=0\) Kaizuka \cite{Ka} generalized corollary \ref{cor2} to all Riemannian symmetric spaces.
\end{itemize}
\end{remark}
We now describe in more details the contents of the paper.\\
Our method of proving  Theorem \ref{main R} follows the one used by Kaizuka \cite{Ka} in  the trivial case. The right hand side of the estimate (\ref{estimate0})  follows from the following uniform continuity estimate for the Helgason-Fourier restriction operator :  There exists a constant \(C>0\) such that  for \(\lambda\in \mathfrak{a}^\ast_{\textit{reg}}\), we have \[(\int_K \mid \widetilde{F}_l(\lambda,k)\mid^2{\rm d}k)^{\frac{1}{2}}\leq C \mid c(\lambda,l)\mid R^{r/2}(\int_G \mid F(g)\mid^2{\rm d}g_K)^{\frac{1}{2}},\] for \(R>1\) and \(F\in L^2(G,\tau_l)\) with \( supp F\subset B(R)\), (see Proposition \ref{uniform} in section 4).\\
On the other hand, the left hand side of (\ref{estimate0}) is a direct consequence of the estimate (\ref{estimate00}).\\
We should mention to the reader that the major difficulty in proving Theorem \ref{main R} is to establish the estimate (\ref{estimate00}).\\ 
Let us briefly recall the method used by Kaizuka to prove the estimate (\ref{estimate00}) in the trivial case (\(l=0)\). It is based on an asymptotic formula for the Poisson transform.  This leads to establish a uniform estimate for the elementary spherical function \(\varphi_\lambda\) on the positive Weyl chamber: For any \(\lambda\in \mathfrak{a}^\ast_{\textit{reg}}\), there exists a positive constant \(C_\lambda\) such that for any \(H\in \overline{\mathfrak{a}^+}\) we have  
\begin{equation}\label{sph1}
\mid \varphi_\lambda({\rm e}^H)-\sum_{s\in W}c(s\lambda){\rm e}^{(is\lambda-\rho)H}\mid\leq C_\lambda{\rm e}^{-(\rho+\tau)(H)},
\end{equation}
The key to prove  (\ref{sph1}) is to use the uniform estimate for \(\varphi_\lambda\) due to Harish-Chandra \cite{H}: There exists a positive constant \(C\) and a non negative integer \(d\) such that for \(\lambda\in \mathfrak{a}^\ast_{\textit{reg}}\) and \(H\in \overline{a^+}\) we have 
\begin{equation}\label{sph2}
\mid \pi(\lambda)\varphi_\lambda({\rm e}^H)\mid \leq C(1+\mid \lambda\mid^{2})^d{\rm e}^{-\rho(H)}.
\end{equation} 
We imitate this method. So we are led to establish  the analogous of the estimate (\ref{sph2}) in the line bundle case  for the elementary spherical function of type \(\tau_{-l}\) 
\begin{align*}
\phi_{\lambda,l}(g)=\int_K {\rm e}^{-(i\lambda+\rho)H(g^{-1}k)}\tau_l(\kappa(g^{-1}k)k^{-1})\,{\rm d}k
\end{align*}
The elementary \(\tau_{-l}\)-spherical function \(\varphi_{\lambda,l}\)  can be expressed by the Heckman-Opdam's hypergeometric function \(F(i\lambda,k(l),a)\), see \cite{HS}. More precisely we have  \[\displaystyle \varphi_{\lambda,l}(a)=\prod_{j=1}^r(\frac{{\rm e}^{t_j}+{\rm e}^{-t_j}}{2})^{-l}F(i\lambda,k(l),a), \quad a=\exp H.\]
Then the desired uniform estimate may be read as follows: There exists a positive constant \(C\) and a non-negative integer \(d\)  such that for \(\lambda\in \mathfrak{a}^\ast\) and \(a\in \overline{A}^+\) we have
\[\mid \pi(\lambda) F(i\lambda,k(l),a)\mid \leq C(1+\mid \lambda\mid^2)^d{\rm e}^{-\rho(l)(H)},\quad a=\exp H
,\]
which is to our knowledge not yet established.\\
In our case the Heckman-Opdam hypegeometric functions \(F(i\lambda,k(l),a)\) are  given explicitly by a determinant of a class Jacobi functions. This allows us to establish the uniform estimate for \(F(i\lambda,k(l),a)\) (The key Lemma) from which we deduce an asymptotic formula for the Poisson transform. The last section is devoted to the proof of (ii) and (iii).
\section{Notation and Preliminaries}
Let ${\C}^{2r+b}$ be endowed with a hermitian form $\displaystyle J(z,w)=\sum_{j=1}^{r} z_j\overline{w_j}- \sum_{j=r+1}^{2r+b} z_j\overline{w_j}$. Let \(G\) be the group all complex \((2r+b)\times(2r+b)\) matrices with determinant \(1\) keeping invariant  the hermitian form $J$.\\
Let $Gr_r({\C}^{2r+b})$ denote the complex Grassmann manifold of all complex $r$-planes in ${\C}^{2r+b}$. The natural action of $G$ on ${\C}^{2r+b}$ gives rise to an action of $G$ on
$$
X_{r,r+b}=\{L\in Gr_r({\C}^{2r+b}); J_{\mid L}\quad \textit{is positive definite}\}.
$$  
In fact $G$ acts transitively on $X_{r,r+b}$ and as a homogeneous space we have the identification $G/K=X_{r,r+b}$, where $K$ the isotropy subgroup of $E_0={\C}^r\times\{0\}$ is a maximal compact subgroup of $G$. Explicitly $ \displaystyle K=S(U(r)\times U(r+b))$ consisting of pairs \((A,D)\) of unitaries with \(\det(AD)=1\).

The Lie algebra  \(\mathfrak{g}\) of $G$   has the Cartan decomposition \(\displaystyle \mathfrak{g}=\mathfrak{k}\oplus \mathfrak{p}\),  where  \(\mathfrak{k}\) is the Lie algebra of $K$, orthogonal with respect to the Killing form \(B:\mathfrak{g}\times\mathfrak{g}\rightarrow {\R} \) given by \(B(X,Y)=2(2r+b)Tr(XY)\). The vector space  \(\mathfrak{p}\) consists of block matrices of the form 
\[
\begin{pmatrix}
0_{r\times r}&Z\\
Z^\ast&0_{r+b\times r+b}
\end{pmatrix}, \quad Z \quad \textit{complex} \, \, r\times(r+b) \,\,  \textit{matrix},
\]
Let \(\mathfrak{a}\) be  the subspace of \(\mathfrak{p}\) consisting of the matrices \(H_T\) of the form 
\begin{equation*}
H_T=\begin{pmatrix}
0_{r\times r}&\textit{diag}\,T &0_{r\times b}\\
\textit{diag}\,T&0_{r\times r}&0_{r\times b}\\
0_{b\times r}& 0_{b\times r}&0_{b\times b}
\end{pmatrix},  t_j\in {\R},
\end{equation*}
where \(\textit{diag}\,T\) denotes the \(r\times r\) matrix with diagonal \(T=(t_1,...,t_r)\), \(t_j\in {\R}\). Then \(\mathfrak{a}\) is a maximal abelian subspace of \(\mathfrak{p}\). Correspondingly the group \(A\) consists of the matrices
\[
a_T=\begin{pmatrix}
diag(\cosh t_{1},\cdots,\cosh t_r)& diag(\sinh t_1,\cdots,\sinh t_r)&O_{r\times b}\\
diag(\sinh t_1,\cdots,\sinh t_r)&diag(\cosh t_{1},\cdots,\cosh t_r)&O_{r\times b}  \\
O_{b\times r}&O_{b\times r}& I_{b}
\end{pmatrix}
.\]
Define \(\alpha_j\in \mathfrak{a}^\ast\) by \(\alpha_j(H_T)=t_j\). The restricted root system \(\Sigma=\Sigma(\mathfrak{g},\mathfrak{a})\) of \(\mathfrak{g}\) in \(\mathfrak{a}\)  is of type  \(BC_r\) or \(C_r\) and consists of \(\pm 2\alpha_j, (1\leq j\leq r), \pm\alpha_i\pm\alpha_j, (1\leq i\neq j\leq r)\) and \(\pm \alpha_j, (1\leq j\leq r)\) with multiplicities \(1, 2\) and \(2b\) respectively
 (\(b\) may equals \(0\), in this case \(X\) is said to be of tube type). \\
We identify $\mathfrak{a}$ with ${\R}^r$ via $\displaystyle H_T\longmapsto T$.\\ We also identify $\mathfrak{a}^\ast$ with ${\R}^r$ via  $\lambda \longmapsto (\lambda_1,\cdots, \lambda_r)$ for $\lambda\in \mathfrak{a}^\ast_c$ given by  $\lambda=\displaystyle\sum_{j=1}^r\lambda_j\alpha_j$ . \\As usual we define the inner product  $<\, ,\, >$ on  $\mathfrak{a}^\ast$ by $\displaystyle <\lambda,\mu>=\frac{1}{4(2r+b)}\sum_{j=1}^r\lambda_j\mu_j$. For $\lambda\in\mathfrak{a}^\ast$ we put $\mid\lambda\mid=\sqrt{<\lambda,\lambda>}$.\\ 
The Weyl group  of $(\mathfrak{g},\mathfrak{a})$ is 
\(\displaystyle W=\{s\mid s(t_1,\cdots, t_r)=(\epsilon_1 t_{\sigma(1)},\cdots, \epsilon_r t_{\sigma(r)}); \sigma\in S_r, \epsilon_i=\pm 1\}\), i.e., \(W\) is the semidirect product of the permutation group \(S_r\), acting by \(t_j\rightarrow t_{\sigma (j)}\), and the group \(\{\pm 1\}^r\), acting by \(t_j\rightarrow \epsilon_jt_j\). We note that \(\mid W\mid =2^rr!\)\\ 
The subset \(\mathfrak{a}_{reg}\subset\mathfrak{a}\)  of regular elements is the complement of the union of the hyper-planes \(\{\beta=0\}, \beta\in \Sigma\). In our identification of $\mathfrak{a}$ with ${\R}^r$, \(\mathfrak{a}_{reg}\) corresponds to the set
\begin{equation*}
\{T=(t_1, \cdots, t_r)\in {\R}^r; t_j\neq 0, t_j\pm t_k\neq 0; \, 1\leq j\leq r, \, 1\leq j\neq k\leq r\}
\end{equation*}
We choose as  positive Weyl chamber \(\mathfrak{a}^+\)(i.e., a connected component of \(\mathfrak{a}_{reg}\) in \(\mathfrak{a}\))
\begin{equation*}
\mathfrak{a}^+=\{H_T\in \mathfrak{a}: t_1>t_2>\dots>t_r>0\},
\end{equation*}
so that the set of corresponding positive roots \(\Sigma^+\) is given by
\[
\Sigma^+=\{ \alpha_j, 2\alpha_j, \alpha_i\pm\alpha_j; \, (1\leq i<j\leq r), 1\leq j\leq r \}.\]
Then the weighted  half-sum of positive roots is \(\displaystyle\rho=\displaystyle\sum_{j=1}^r(1+b+2(r-j))\alpha_j\).\\
Put \(\displaystyle\beta_i=\alpha_i-\alpha_{i+1} (1\leq i\leq r-1)\) and \(\displaystyle\beta_r=\delta\alpha_r\). Then the set of positive simple roots is 
 \(\displaystyle\Psi=\{\beta_1, \cdots, \beta_r\}\),
with \(\delta=2\) in the tube case and \(\delta=1\) in the non tube case.\\
Let \(\displaystyle\mathfrak{n}=\sum_{\alpha\in\Sigma^+} \mathfrak{g}_\alpha\) and let \(N\) be the analytic subgroup of \(G\) corresponding to the subalgebra  \(\mathfrak{n}\). Then \(G=KAN\) is an Iwasawa decomposition of \(G\) and each \(g\in G\) can be uniquely written
\(g=\kappa(g){\rm e}^{H(g)}n(g)\), with  \(\kappa(g)\in K,  H(g)\in \mathfrak{a}\) and \( n(g)\in N\).\\
Let \(A^+=\exp(\mathfrak{a}^+)\), and \(\overline{A^+}\) the closure of \(A^+\) in \(G\). Then we have the polar  decomposition  \(G=K\overline{A^+}K\) of \(G\) and for each  \(g\in G\) there exists a unique element \(A^{+}(g)\in \overline{\mathfrak{a}^{+}}\) such that \(g\in K{\rm e}^{A^{+}(g)}K\).\\
In our identification of $\mathfrak{a}$ with ${\R}^r$, the closed chamber $\overline{\mathfrak{a}}^+$ corresponds to the set 
$$
\mathcal{C}=\{T=(t_1,\cdots, t_r)\in {\R}^r; t_1\geq t_2\geq\cdots\geq t_r\geq0\}.
$$
We shall use the letter \(C\) to denote any positive constant independent of the parameter \(\lambda\). Occasionally \(C\) may be suffixed to show it dependency on some parameters.
\subsection{Integral formulas}
The Lebesgue measure  on \(\mathfrak{a}\) and \(\mathfrak{a}^\ast\) is normalized  so that the Fourier transform 
\begin{equation*}
\mathcal{F}_{\mathfrak{a}}\phi(\lambda)=\int_{\mathfrak{a}}{\rm e}^{-i\lambda(H)} \phi(H){\rm d}H, \quad \phi\in \mathcal{S}(\mathfrak{a}), \lambda\in \mathfrak{a}^\ast
\end{equation*}
has as its inverse
\begin{equation*}
\phi(H)=\int_{\mathfrak{a}^\ast}{\rm e}^{i\lambda(H)} \mathcal{F}_{\mathfrak{a}}\phi(\lambda) {\rm d}\lambda.
\end{equation*}
In the above $\mathcal{S}(\mathfrak{a})$ denotes the Schwartz space on $\mathfrak{a}$.\\
Consider the function \(\Delta\)  on $A^+$ defined by 
\begin{equation*}
\Delta({\rm e}^{H_T})=\omega^{2}({\rm e}^{H_T})2^{r(2b+1)}\prod_{j=1}^r\sinh^{2b} t_j\sinh2t_j,
\end{equation*}
where \(\omega({\rm e}^{H_T})\)  is the Weyl denominator given by
\begin{equation*}
\omega({\rm e}^{H_T})=2^{\frac{r(r-1)}{2}}\prod_{1\leq j<k\leq r}(\cosh 2t_j-\cosh 2t_k).
\end{equation*}
We normalize the invariant measure \({\rm d}g_K\) by  
 \begin{equation*}
\int_{G/K} h(gK){\rm d}g_K=\int_G h(g.0)\, {\rm d}g=\int_K\int_{A^+}h(ka)\Delta(a)\,{\rm d}k\,{\rm d}a, \quad h\in C_c(G/K).
\end{equation*}
 \subsection{The Helgason-Fourier transform}
In this section we collect some known facts on the Helgason-Fourier transform on homogeneous line bundles over \(SU(r,r+b)/S(U(r)\times U(r+b))\) referring to \cite{Sh}, \cite{Ca} for more details.\\
Let \(Z\) be the basis of the center of \(\mathfrak{k}\) given by \(\displaystyle Z=\begin{pmatrix}
\frac{i}{r}I_r&0\\
0&-\frac{i}{r+b}I_{r+b}
\end{pmatrix}\).
Then the group \(\displaystyle K=SU(r)\times SU(r+b)\exp \mathbb{R}Z\) is a direct product so that any  character of \(K\) is trivial on \(SU(r)\times SU(r+b)\). Therefore every character of \(K\) is parameterized by \(l\in{\Z}\) and is given by \(\displaystyle \tau_l(k_s\exp tZ)={\rm e}^{ilt}\).\\
The  Helgason-Fourier transform of \(F\in C_c^\infty(G,\tau_l)\) is the function from $\mathfrak{a}^\ast_c\times K$  defined by
\begin{equation*}
 \widetilde{F}_l(\lambda,k)=\int_G {\rm e}^{(i\overline{\lambda}-\rho)H(g^{-1}k)}\tau_l^{-1}(\kappa(g^{-1}k)) F(g)\, {\rm d}g.
 \end{equation*}
Let  \(h_j\) \((0\leq j\leq r\)) be \(r+1\)-nonconjugate \(\Theta\)-stable Cartan subalgebras built from the following subsets 
\(\theta_j\)of \(\Psi\)
\[\theta_j=\{\beta_i; r-j+1\leq i\leq r\},\]
 and let \(P_j\) be the corresponding cuspidal parabolic subgroup with Langlands decomposition \(P_j=M_jA_jN_j\). Then the following inversion formula holds (see \cite{Sh})
\begin{align*}
 & F(x)=\frac{1}{\mid W \mid}\int_{\mathfrak{a}^\ast}\int_K {\rm e}^{-(i\lambda+\rho)H(x^{-1}k)}\tau_l(\kappa(x^{-1}k))\widetilde{F}_l(\lambda,k)\mid c(\lambda,l)\mid^{-2} {\rm d}k\,{\rm d}\lambda +\\
&\sum_{j=1}^{r-1}d_j\sum_{-\lambda^{''}\in D_{l,j}}\int_{\mathfrak{a}_j^\ast}\int_K{\rm e}^{(-(i\lambda'+\lambda^{''})-\rho)H(x^{-1}k)}\tau_l(\kappa(x^{-1}k))\widetilde{F}_l(\lambda^{'}-i\lambda^{''},k)\mid c^j(\lambda,l)\mid^{-2}{\rm d}k{\rm d}\lambda^{'}+\\
&\sum_{i\lambda\in D_{l,r}}d_r(i\lambda,l)\int_K{\rm e}^{(-(i\lambda{'}+\lambda^{''})-\rho)H(x^{-1}k)}\tau_l(\kappa(x^{-1}k))\widetilde{F}_l(\lambda^{'}-i\lambda^{''},k){\rm d}k,\qquad \qquad (\lambda=\lambda'+i\lambda")
\end{align*} 
and the Plancherel formula 
\begin{align}\label{P}
&\int_G\mid F(g)\mid^2{\rm d}g=\frac{1}{\mid W \mid}\int_{\mathfrak{a}^\ast}\int_K\mid \widetilde{F}_l(\lambda,k)\mid^2 \mid c(\lambda,l)\mid^{-2}{\rm d}\lambda{\rm d}k+ \\
&+\sum_{j=1}^{r-1} d_j\sum_{-\lambda^{''}\in D_{l,j}}\int_{\mathfrak{a}_j^\ast}\int_K \widetilde{F}_l(\lambda^{'}-i\lambda^{''},k)\overline{\widetilde{F}_l(\lambda^{'}+i\lambda^{''},k)}\mid c^j(\lambda,l)\mid^{-2}{\rm d}k{\rm d}\lambda^{'}\notag\\
&+\sum_{i\lambda\in D_{l,r}}d_r(i\lambda,l)\int_K\widetilde{F}_l(\lambda^{'}-i\lambda^{''},k)\overline{\widetilde{F}_l(\lambda^{'}+i\lambda^{''},k)}{\rm d}k\notag
\end{align}
Here \(\mid c^j(\lambda,l)\mid^{-2}{\rm d}\lambda^{'}\) is the Plancherel measure associated with the principal \(P_j\) series \(j=1,\cdots,r-1\), (the parameter \(\lambda\) in \cite{Sh} corresponds here to \(i\lambda\)).\\ In the above $d_j, d_r(i\lambda,r)$ are some constants and  $D_{\lambda,l}$ is the set of parameters  of the relative discrete series. 
The function  \(c(\lambda,l)\) is a meromorphic function on \(\mathfrak{a}^\ast_c\simeq {\C}^r\) given by 
\[c(\lambda,l)=\frac{B}{\prod_{1\leq i<j\leq r}\lambda_i^2-\lambda_j^2}\prod_{j=1}^r \dfrac{2^{-i\lambda_{j}}\Gamma(i\lambda_{j})}{\Gamma(\dfrac{1}{2}(b+1+i\lambda_{j}+l))\Gamma(\dfrac{1}{2}(b+1+i\lambda_{j}-l))},\] 
with \(\displaystyle B=(-1)^{r(r-1)/2}2^{r(2r-1+b))}(b!)^{r}\prod_{j=1}^{r-1}(b+j)^{r-j}j!\).\\
Now, for any \(F\in C_c^\infty(G,\tau_l)\), define its Radon transform \(\mathcal{R}F\) by 
\begin{equation*}
\mathcal{R}F(g)={\rm e}^{\rho(H(g))}\int_{N}F(gn){\rm d}n.
\end{equation*}
Then a simple calculation shows that the Fourier transform may be written as
\[ \widetilde{F}_l(\lambda,k)=\int_A {\rm e}^{(-i\lambda+\rho)(H)}{\rm d}H\int_N F(k{\rm e}^H n){\rm d}n.\]
Set \(\mathcal{R}F(H,k)=\mathcal{R}F(k{\rm e}^H)\).\\
For the Radon transform we have the following support theorem.\\
As is well known, the Killing form induces a norm \(\vert\, \vert\)   on \(\mathfrak{p}\) and a distance function \(d\) on \(G/K\). 
Moreover, if \(g\) is written in the form \(g=k{\rm e}^X\) with respect to the Cartan decomposition \(G=K\exp \mathfrak{p},\) \((k\in K, X\in \mathfrak{p})\), then \(\displaystyle d(0,g.0)=\mid X\mid\), where \(0=eK\).\\
Define the open ball centered at $0$ and of radius $R>0$ by \(B(R)=\{gK\in G/K; d(0,g.0)<R\}\). Let  \(\mathfrak{a}(R)=\{H\in\mathfrak{a}; \mid H\mid< R\}\)
\begin{lemma}\label{Radon}
Let \(\displaystyle F\in C_c^\infty(G,\tau_l)\).
\begin{enumerate}
\item[(i)] If \(\textit{supp}F\subset \overline{B(R)}\), then \(\textit{supp} \mathcal{R}F\subset \overline{\mathfrak{a}(R)}\times K\).
\item[(ii)] The Radon transform and the Fourier transforms are related by the following formula \[\widetilde{F}_l(\lambda,k)=\mathcal{F}_{a}[\mathcal{R}F(k, .)](\lambda).\]
\end{enumerate}
\end{lemma}
\begin{proof}
We have only to prove (i). We know that  \[d(0,{\rm e}^Hn.0)\geq \mid H\mid , \quad H\in \mathfrak{a}, n\in N,\] and since the distance function is \(K\)-invariant it follows that \(\displaystyle d(0,k{\rm e}^Hn.0)\geq \mid H\mid \) for all \(k\in K, H\in \mathfrak{a}, n\in N\).
Therefore \(\mathcal{R}F\) has support  in \(\overline{\mathfrak{a}(R)}\times K\) if \(f\) has support in \(\overline{B(R)}\).
\end{proof}
\subsection{The Poisson transform on line bundle}
Let \(M\) be the centralizer of \(A\) in \(K\). Then \(M\) is of the form \[ M=\Bigg\{\begin{pmatrix}
diag ({\rm e}^{i\theta_1},\cdots,{\rm e}^{i\theta_r})&0\\
0&diag ({\rm e}^{i\theta_1},\cdots,{\rm e}^{i\theta_r})
\end{pmatrix}; \theta_j\in {\R},\theta_1+\cdots \theta_r\in \pi{\Z}\Bigg\},\quad \textit{if} \quad b=0,\] 
and of the form \[ M=\Bigg\{\begin{pmatrix}
diag ({\rm e}^{i\theta_1},\cdots,{\rm e}^{i\theta_r})&0&0\\
0&diag ({\rm e}^{i\theta_1},\cdots,{\rm e}^{i\theta_r})&0\\
0&0&L
\end{pmatrix}; \theta_j\in {\R}, L\in U(b); {\rm e}^{i2(\theta_1+\cdots+\theta_r)}\det L=1\Bigg\},\]
if \(b\neq 0\).\\ 
Let \(G\times_P {\C}\) be  the homogeneous line bundle  associated with the character of \(P=MAN\) defined by \(man\rightarrow a^{\rho-i\lambda}\tau_l(m)\).
We identify the space of hyperfunction valued sections of \(G\times_P{\C}\) with the space
\(\mathcal{B}(G,\tau_l\otimes\rho-i\lambda\otimes 1)\) consisting of all hyperfunctions \(f\) on \(G\) that satisfy
\(f(gman)=a^{i\lambda-\rho}\tau_l(m)^{-1}f(g), \textit{for} \,\, g\in G, m\in M, a\in A, n\in N\).
Then the Poisson transform is the \(G\)-map defined from \(\mathcal{B}(G,\tau_l\otimes\rho-i\lambda\otimes 1)\)  to \(C^\infty(G,\tau_l)\)
by
\begin{align*}
P_{\lambda,l}f(g)=\int_K \tau_l(k)f(gk)\,{\rm d}k.
\end{align*}
By the Iwasawa decomposition, the restriction from \(G\) to \(K\) gives an isomorphism from \(\mathcal{B}(G,\tau_l\otimes\rho-i\lambda\otimes 1)\) onto \(\mathcal{B}(K,\tau_l)\) the space of all hyperfunctions \(f\) on \(K\) such that \(\displaystyle f(km)=\tau_l(m)^{-1}f(k), \textit{for}\,\,  k\in K,  m\in M\).\\
Then a direct calculation shows that the Poisson transform of \(f\in \mathcal{B}(K,\tau_l)\) is given by 
\[ P_{\lambda,l}f(g)=\int_K {\rm e}^{-(i\lambda+\rho)H(g^{-1}k)}\tau_l(\kappa(g^{-1}k))f(k)\,{\rm d}k.\]
Recall from the introduction that \( P_{\lambda,l}f\) are joint eigenfunctions of the commutative algebra \(\mathbb{D}(E_l)\), as \(f\) runs \(\mathcal{B}(K,\tau_l)\). More precisely let \(\gamma_l\) denote the Harish-Chandra isomorphism from \(\mathbb{D}(E_l)\) onto \(\mathcal{S}(\mathfrak{a}_c)^W\) the set of \(W\)-invariant elements in the symmetric algebra \(\mathcal{S}(\mathfrak{a}_c)\), see \cite{Sh1} for detailed discussion on the subject.\\ Let \(\mathcal{E}_{\lambda,l}(G)\) denote the solution space of the system of differential equations on \(C^\infty(G,\tau_l)\)
\[DF=\gamma_l(D)(\lambda)F, \quad D\in \mathbb{D}(E_l).\]
Since \(\mathbb{D}(E_l)\) contains an elliptic element coming from the Casimir element, \(\mathcal{E}_{\lambda,l}(G)\) consists of real analytic functions.\\
Now we recall a result due to Shimeno \cite{Sh} on the characterization of the image of the Poisson transform, which reads in our case as follows
\begin{theorem}\cite{Sh}\label{Shimeno}
Let \(l\in{\Z}\), \(\lambda\in \mathfrak{a}^\ast_c\simeq {\C}^r\) satisfy the conditions 
\[-2i\lambda_j\notin \{1,2,\cdots\}\,  (1\leq j\leq r), \quad i(\lambda_k\pm \lambda_j)\notin \{1,2,\cdots\} \,  (1\leq k<j\leq r),\]
and \[b+1+i\lambda_j\pm l\notin 2{\Z}^-,\]
then the Poisson transform \(P_{\lambda,l}\) is an isomorphism from \(\mathcal{B}(K,\tau_l)\) onto \(\mathcal{E}_{\lambda,l}(G)\).
\end{theorem}
\section{Uniform estimate for the \(\tau_{-l}\)-elementary spherical functions }
In this section we prove a  uniform estimate for  the \(\tau_{-l}\)-elementary spherical functions \(\varphi_{\lambda,l}\).\\ Recall from [\cite{Sh1}, Proposition 4.6] when \(G=SU(r,r+b)\) that  the functions of the form
\begin{equation*}\label{spherical}
\varphi_{\lambda,l}(g)=\int_K {\rm e}^{-(i\lambda+\rho)H(g^{-1}k)}\tau_l(\kappa(g^{-1}k)k^{-1})\,{\rm d}k
\end{equation*}
exhaust the class of the elementary spherical functions of type \(\tau_{-l}\), with \(\lambda\in \mathfrak{a}^\ast_c\).\\
Since \(G=K\overline{A^+}K\), \(\varphi_{\lambda,l}\) is completely defined by its restriction to \(A^+\). Moreover \(\varphi_{\lambda,l}\) can be expressed by the  Heckman-Opdam's hypergeometric functions as follows:\\
For \(l\in {\Z}\), let \(m_\alpha(l)\) be a deformation of root multiplicities defined by 
\[m_{\alpha_j}(l)=2b+2l, \quad m_{\alpha_j\pm \alpha_k}(l)=2 \quad \textit{and} \quad m_{2\alpha_j}(l)=1-2l,\] and let 
\(
\rho(l)=\rho-l\displaystyle\sum_{j=1}^r\alpha_j
\).\\
Put \(R=2\Sigma\) and define the multiplicities function \(k(l)\) on \(R\) by \(k_{2\alpha}(l)=\frac{1}{2}m_\alpha(l)\).\\
Let \(F(\lambda,k(l),a)\) be the  Heckman-Opdam hypergeometric function associated with \(R\) and the multiplicities \(k(l)=(b+l,1,1/2-l)\).  Then by [\cite{HS}, Theorem 5.2.2] the \(\tau_{-l}\)-elementary spherical function \(\varphi_{\lambda,l}\) is given by
\begin{align*}
\varphi_{\lambda,l}(a_T)=\prod_{j=1}^r(\frac{{\rm e}^{t_j}+{\rm e}^{-t_j}}{2})^{-l}F(i\lambda,k(l),a_T),\quad a_T={\rm e}^{H_T}.
\end{align*}
Moreover the above hypergeometric function of Heckman-Opdam can be expressed  in terms of the Jacobi functions
\[\phi_{\mu}^{(\alpha,\beta)}(t)=F(\frac{\alpha+\beta+1+i\mu}{2},\frac{\alpha+\beta+1-i\mu}{2}, \alpha+1; -\sinh^2 t).
\]
More precisely, put \(\displaystyle u(a_T)=\prod_{j=1}^r({\rm e}^{t_j}+{\rm e}^{-t_j})\), then  for $\lambda\in \mathfrak{a}^{*}_{\C}$ we have 
\begin{eqnarray}\label{spherical function} 
\varphi_{\lambda,l}(a_T)&=& u(a_T)^{-l}\frac{B_1}{\prod\limits_{1\leq i<j \leq r}(\lambda_{i}^{2}-\lambda_{j}^{2})}\dfrac{\det (\phi^{(b,-l)}_{\lambda_i}(t_j))}{\omega(a_T)},
\end{eqnarray}
with $\displaystyle B_1=(-1)^{r(r-1)/2}2^{r(2(r-1)+l)}\prod_{j=1}^r(b+j)^{r-j}j!$\\
This follows from [\cite{Sh2}, Theorem 2.5], with \(k_s=b+l, k_m=1\) and \(k_l=1/2-l\).\\
Now to state and prove the main result of this section let  us introduce the polynomial function \(\displaystyle \pi:\overline{\mathfrak{a}^+}\rightarrow {\R}$ given by
\begin{align*}
 \pi(\lambda)=(\frac{1}{4(b+2r)})^{r^2}\prod_{j=1}^r\lambda_j\prod_{1\leq i<j\leq r}(\lambda_i^2-\lambda_j^2)\quad \lambda=\sum_{j=1}^r\lambda_j \alpha_j
\end{align*}
\begin{k lemma}
Let \(l\in{\Z}\). There exist a positive constant \(C\) and a \\non-negative integer \(d\) such that   for \(\lambda \in a^{*}\)
and \(a_T \in \overline{A^+}\) we have
\begin{eqnarray}\label{estimate1}
\left|\pi(\lambda)\varphi_{\lambda,l}(a_T)u(a_T)^{l}\right| &\leq& C (1+\mid \lambda\mid^2)^{d}e^{-\rho(l)(H_T)}, \quad a_T={\rm e}^{H_T}.
\end{eqnarray}
\end{k lemma}
For the proof of the Key Lemma, we shall need the following estimate on the Jacobi functions.
\begin{lemma A1}
For \(\alpha ,\beta \in \R \) such that \(\alpha>-1/2\)  and  \(n\in \N\) there exists a positive constant \(C\)  such that
\begin{align}\label{estimate J}
\mid \lambda\dfrac{d^{n}}{dt^{n}}\phi_{\lambda}^{\alpha,\beta}(t)\mid\leq C(1+\lambda^{2})^{n+1}e^{-(\alpha+\beta+1) t},
\end{align}
for all \(t>0\) and \(\lambda\in{\R}\).
\end{lemma A1}
The proof of this Lemma is postponed to  the last section.\\
Now we give the proof of the Key lemma.
\begin{proof} We will proceed by induction on the rank \(r\).\\
First of all, since the function  \(a_T \rightarrow \varphi_{\lambda,l}(a_T) \) is continuous it is enough to prove the Key lemma for  $a_T\in A^{+}$  with \(t_1>1\) .\\
Note that \begin{align*}
\pi(\lambda)\varphi_{\lambda,l}(a_T)u^{l}(a_T)=\frac{B_1 \det (\lambda_i\phi^{b,-l}_{\lambda_i}(t_j))}{\omega(a_T)},
\end{align*}
and the denominator 
\[\omega(a_T)={\rm e}^{2\sum_{j=1}^{r}(r-j)t_j}\prod_{1\leq i<j\leq r}(1-e^{-2(t_i-t_j)})(1-e^{-2(t_i+t_j)}),\]
so we have to show that there exists \(C>0\) such that for \(\lambda\in \mathfrak{a}^\ast\) and \( H_T\in \mathfrak{a}^+\) with $t_1>1$ 
\begin{align}\label{estimate x}
 \frac{\mid \det (\psi_{\lambda_i}(t_j))\mid}{\prod_{1\leq i<j\leq r}(1-e^{-2(t_i-t_j)})(1-e^{-2(t_i+t_j)})}\leq C(1+\mid \lambda\mid)^d{\rm e}^{-\sum_{j=1}^r(b+1-l)t_j}, 
\end{align}
where we have put \(\psi_{\lambda_i}(t_j)=\lambda_i\phi^{b,-l}_{\lambda_i}(t_j)\).\\
We consider the rank two case (\(r=2\)). Noting that \(1-{\rm e}^{-2(t_1+ t_2)}>1-{\rm e}^{-2}\)( \(t_1>1\))  we have   
\begin{align*}
 \frac{\mid \det (\psi_{\lambda_i}(t_j))\mid}{(1-e^{-2(t_1-t_2)})(1-e^{-2(t_1+t_2)})}\leq (1-{\rm e}^{-2})^{-1} \frac{\mid \det (\psi_{\lambda_i}(t_j))\mid}{(1-e^{-2(t_1-t_2)})}
\end{align*}
If \(t_1>t_2+1\) then \((1-2{\rm e}^{-2(t_1-t_2)})>1-{\rm e}^{-2}\) and the estimate (\ref{estimate x}) follows from (\ref{estimate J}).\\ 
When \(t_{2}<t_{1}\leq 1+t_{2}\) write  the determinant as 
\begin{align*}
\det (\psi_{\lambda_{i}}(t_j))_{1\leq i,j\leq2}=\left| \begin{array}{cc}
\psi_{\lambda_{1}}(t_{1})-\psi_{\lambda_{1}}(t_{2})\quad & \psi_{\lambda_{1}}(t_{2})\\
\psi_{\lambda_{2}}(t_{1})-\psi_{\lambda_{2}}(t_{2})\quad & \psi_{\lambda_{2}}(t_{2})
 \end{array}\right|.
\end{align*}
From this (\ref{estimate x}) follows easily. Indeed, using again (\ref{estimate J}), we have
\begin{eqnarray*}
\mid\psi_{\lambda_i}(t_1)-\psi_{\lambda_i}(t_2)\mid &\leq & C(1+ \lambda_i^{2})^{2}(t_1-t_2)\sup_{s\in [t_{2},t_{1}]}{\rm e}^{-(b+1-l)s} \quad i=1,2\\
&\leq &C(1+ \lambda_i^{2})^{2}(t_1-t_2)\sup_{s\in [t_{2},t_{1}]}{\rm e}^{-(b+1-l)(s-t_{1})}e^{-(b+1-l)t_{1}} ,\quad i=1,2\\
&\leq &C(1+\lambda_i^{2})^{2}(t_1-t_2)\left(\sup_{s\in [-1,0]}{\rm e}^{-(b+1-l)s}\right)e^{-(b+1-l)t_{1}}\quad i=1,2\\
&\leq & C(1+\lambda_i^{2})^{2}(t_1-t_2)e^{-(b+1-l)t_{1}}\quad i=1,2 
\end{eqnarray*}
Now, noting that the function  \(s\mapsto \frac{s}{1-e^{-s}}\) is bounded on \([0,1]\) we get 
\begin{align*}
\frac{|\det [(\psi_{\lambda_{i}}(t_j))_{1\leq i,j\leq2}]|}{1-e^{-(t_1-t_2)}}\leq C (1+ \mid\lambda\mid^{2})^{2}{\rm e}^{-(b+1-l)(t_1+t_2)}.
\end{align*}
This finishes the proof of the estimate (\ref{estimate x}) for \(r=2\).\\
Next, assume the assertion for \(r-1\).\\
\textbf{Case 1:} Suppose that there exists j \(\in \{2,..,r\}\) such that \(t_1>t_j+1\).\\
Let p the smallest integer among \(j \in \{2,..,r\}\), such that \(t_1>t_j+1\), then 
\[(\quad t_1-t_j>1, \forall j \geq p \quad and \quad t_1-t_j\leq 1, \forall  2\leq j \leq p-1  ).\]
Since \( t_1\pm t_j>1\) for \(p\leq j\leq r\) we get
\begin{align*}
\frac{\mid \det (\psi_{\lambda_i}(t_j))\mid}{\prod_{1\leq i<j\leq r}(1-e^{-2(t_i-t_j)})(1-e^{-2(t_i+t_j)})}\leq \frac{(1-{\rm e}^{-2})^{-2r+p}\mid \det (\psi_{\lambda_i}(t_j))\mid}{\prod_{j=1}^{p-1}(1-{\rm e}^{-2(t_1-t_j)})\prod_{2\leq i<j\leq r}(1-e^{-2(t_i-t_j)})(1-e^{-2(t_i+t_j)})}
\end{align*}
Now, let us  introduce the interpolation  polynomials \\
\[\left\{\begin{array}{ccc}L_{\lambda_i}(t):&=&\sum_{k=2}^{p-1}\psi_{\lambda_i}(t_k)l_k(t), \quad i=1,\cdots, r, \quad if \quad p\geq 3\\
L_{\lambda_{i}}(t)&=&0,\qquad \qquad \quad \quad \quad i=1,\cdots, r, \quad if \quad p=2\end{array}\right.\]
where \(l_k\) is the Lagrange basis polynomials associated to the points \(t_2,\cdots, t_{p-1}\),
\begin{align*}l_{j}(t)= \prod_{i=2,i\neq j}^{p-1}\dfrac{t-t_i}{t_j-t_i}, \quad j=2,...,p-1, \quad p>3,
 \end{align*}
and \( l_{2}(t)=1 \quad if\quad p=3\).\\
Next, in \(\det (\psi_{\lambda_i}(t_j))\) replace the first column by \(F_{\lambda_i}(t_1)= \psi_{\lambda_i}(t_1)-L_{\lambda_i}(t_1)\).
We have 
\begin{align*}
\det (\psi_{\lambda_i}(t_j))=\sum_{i=1}^r(-1)^{i+1}F_{\lambda_i}(t_1)\Delta_i(t_2,\cdots,t_r),
\end{align*}
where 
\begin{align*}
\Delta_i(t_2,..,t_r)=
\left| \begin{array}{cccc}
\psi_{\lambda_{1}}(t_{2}) & \ldots & \psi_{\lambda_{1}}(t_{r}) \\
\vdots &  & \vdots \\
\psi_{\lambda_{i-1}}(t_2) & \ldots & \psi_{\lambda_{i-1}}(t_{r})\\
\psi_{\lambda_{i+1}}(t_2) &  \ldots & \psi_{\lambda_{i+1}}(t_{r})\\  
\vdots &  & \vdots \\
\psi_{\lambda_r}(t_2) &  \ldots & \psi_{\lambda_r}(t_{r})\\  
\end{array}\right|.
\end{align*}
Thus for some $d_i\in \N,i=1,..,r$
\begin{align*}
\frac{\mid \det (\psi_{\lambda_i}(t_j))\mid}{\prod_{1\leq i<j\leq r}(1-e^{-2(t_i-t_j)})(1-e^{-2(t_i+t_j)})}\leq C\frac{\sum_{i=1}^r\mid F_{\lambda_i}(t_1)\mid (1+\mid\lambda'_i\mid^{2})^{d}}{\prod_{j=2}^{p-1}(1-e^{-2(t_1- t_j)})}{\rm e}^{-\sum_{j=2}^r(b+1-l)t_j},
\end{align*} 
by the induction hypothesis with $\lambda'_i=(\lambda_1,...,\lambda_{i-1},\lambda_{i+1},...,\lambda_{n})$.\\
Next, since the function \(t_1 \rightarrow F_{\lambda_i}(t_1)\) vanishes at \(t_1=t_2=..=t_{p-1}\), we may use the elementary lemma (see \ref{L}) to get 
\begin{align}\label{app est}|F_{\lambda_i}(t_1)|\leq \prod_{j=2}^{p-1}(t_1-t_j)\sup_{s\in[t_{p-1},t_1]}|F_{\lambda_i}^{(p-2)}(s)|.\end{align}
Noting that \((\frac{d}{ds})^{p-2}F_{\lambda_i}(s) = (\frac{d}{ds})^{p-2}\psi_{\lambda_i}(s)\), and using the estimate (\ref{estimate J}) we easily see that for  \(s\in [t_{p-1},t_1]\) 
 \[|F_{\lambda_i}^{(p-2)}(s)|\leq C(1+\lambda_i^2)^{p-1}e^{-(b+1-l)s}\leq  C(\sup_{s\in [-1,0]}e^{-(b+1-l)s})(1+\lambda_i^2)^{r-1}e^{-(b+1-l)t_1}\]
Putting all together, we get 
\begin{align*}
\frac{\mid \det (\psi_{\lambda_i}(t_j))\mid}{\prod_{1\leq i<j\leq r}(1-e^{-2(t_i-t_j)})(1-e^{-2(t_i+t_j)})}\leq C(1+\mid \lambda \mid^{2})^{d}{\rm e}^{-\sum_{j=1}^r(b+1-l)t_j},
\end{align*}
for some $d\in \N$, as to be shown.\\
\textbf{Case 2 :} For every j \( \in \{ 2,..,r \} \), we have  \( t_j <  t_1 \leq t_j +1\).\\
In \(\det (\psi_{\lambda_i}(t_j))\) we replace the first column by \(F_{\lambda_i}(t_1)= \psi_{\lambda_i}(t_1)-\sum_{j=2}^{r}l_{j}(t_1)\psi_{\lambda_i}(t_j)\)  where $(l_{j})_{j=2,..,r}$ is the lagrange basis of polynomials associated to the points $t_{2},...,t_{r}$. To conclude  we follow the same reasoning we did in the case 1. This finishes the proof of the Key Lemma.
\end{proof}
As an immediate  consequence  we get a uniform estimate for the spherical function of type \(\tau_{-l}\).
\begin{corollary}\label{cor}
Let \(l\in {\Z}\). There exists a positive constant \(C\) and \(d \in \N\) such that for \\\(\lambda\in \mathfrak{a}^\ast_{\textrm{reg}},\ \, H\in \overline{a^+}\)
\begin{equation}\label{estimate6}
|\pi(\lambda)\varphi_{\lambda,l}({\rm e}^H)|\leq C(1+\mid \lambda\mid^2)^d {\rm e}^{-\rho(H)}.
\end{equation}
\end{corollary}
We can now prove the following uniform estimate on the \(\tau_{-l}\)-elementary spherical functions.
\begin{lemma} Let \(l\in{\Z}\) and  \(\lambda\in \mathfrak{a}^\ast_{\textrm{reg}}\). Then  there exists a positive constant \(C_{\lambda,l}\) such that for all \(H_T\in \overline{\textbf{a}^{+}}\) we have
 \begin{eqnarray}\begin{split}\label{estimate7}
 |\varphi_{\lambda,l}({\rm e}^{H_T})-u({\rm e}^{H_T})^{-l}\sum_{s\in W}c(s\lambda,l){\rm e}^{(is\lambda-\rho(l))(H_T)}|
 \leq C_{\lambda,l}e^{-\rho(H_T)}e^{-\tau(H_T)}.
\end{split}\end{eqnarray}
In the above \(\tau(H)=\displaystyle\min_{\alpha\in \Psi}\alpha(H).\)
\end{lemma}
\begin{proof}
By using (\ref{estimate6}) and noting that \({\rm e}^{-\rho(l)(H_T)}(u({\rm e}^{H_T}))^{-l}\leq {\rm e}^{-\rho(H_T)}\) we easily see that
\begin{equation*}\begin{split}
|\varphi_{\lambda,l}(e^{H_T})-u(e^{H_T})^{-l}&\sum_{s\in W}c(s\lambda,l)e^{(is\lambda-\rho(l))(H_T)}|\\
&\leq (C\mid \pi(\lambda)\mid^{-1}(1+\mid \lambda\mid^2)^d+\mid W\mid \mid c(\lambda,l)\mid ) {\rm e}^{-\rho(H_T)},
\end{split}\end{equation*}
for any \( H_T\in \overline{\textbf{a}^{+}}\).\\
If   \( H_T\in \overline{\textbf{a}^{+}}\)  with \(\tau(H_T)<1\), then using the trivial estimate \({\rm e}^{-\tau(H_T)}{\rm e}>1\), we get
\begin{equation*}\begin{split}
|\varphi_{\lambda,l}(e^{H_T})-u(e^{H_T})^{-l}&\sum_{s\in W}c(s\lambda,l)e^{(is\lambda-\rho(l))(H_T)}|\\
&\leq (C\mid \pi(\lambda)\mid^{-1}(1+\mid \lambda\mid^2)^d+\mid W\mid \mid c(\lambda,l)\mid) {\rm e}{\rm e}^{-\tau(H_T)} {\rm e}^{-\rho(H_T)}.
\end{split}\end{equation*}
For \(H_T\) with \(\tau(H_T)\geq 1\), the estimate follows  in the same manner  as in  the proof of  the case \(l=0\), so we omit it, see \cite{Ha}.
\end{proof}
\newpage
\section{Fourier restriction Theorem}
The main result of this section is
\begin{proposition}(Fourier restriction theorem)\label{uniform}
Let \(l\in{\Z}\). There exists a positive constant \(C\) such that for \(\lambda\in \mathfrak{a}_{\textit{reg}}^\ast\) and \( R>1\) we have
\begin{equation}\label{F.uniform}
\bigg(\int_K\mid \widetilde{F}_{l}(\lambda,k)\mid^2\,{\rm d}k\bigg)^\frac{1}{2}\leq C \mid c(\lambda,l)\mid R^\frac{r}{2}\bigg(\int_{G/K}\mid F(g)\mid^2{\rm d}(gK)\bigg)^\frac{1}{2},
\end{equation}
for every  \(F\in L^2(G,\tau_l)\) with \(supp F\subset B(R)\), 
\end{proposition}
To prove Proposition \ref{uniform} we follow mainly the same method used by Kaizuka for \(l=0\) which can be traced back to Anker \cite{A}. To this end we shall need estimates of the Harish-Chandra \(c\)-function.\\
Let \(\pi_l:\mathfrak{a}^\ast\rightarrow {\C}\) be the polynomial function defined by 
\begin{displaymath}
\pi_l(\lambda)=(\frac{1}{4(b+2r)})^{r^2}\left\{\begin{array}{rcl}
\prod_{j=1}^r\lambda_j\prod_{1\leq j<k\leq r}(\lambda_j^2-\lambda_k^2)&\textrm{if \(b+1\pm l\notin 2{\Z}^-\),}\\
\prod_{1\leq j<k\leq r}(\lambda_j^2-\lambda_k^2) &\textrm{if \(b+1\pm l\in 2{\Z}^-\),}
\end{array}\right.\quad \lambda=\sum_{j=1}^r\lambda_j\alpha_j.
\end{displaymath}
Define  the function \(\mathbf{b}(\lambda,l)\) on \(\mathfrak{a}^\ast\) by \(\displaystyle \mathbf{b}(\lambda,l)=\pi_l(i\lambda)c(\lambda,l)\).
\begin{lemma} Let \(l\in{\Z}\).
\begin{itemize}
\item[(i)] The function \(\mathbf{b}(.,l)\) has no zero in \(\mathfrak{a}^\ast\).
\item[(ii)] There exists a positive constant \(C\) such that for \(\lambda\in \mathfrak{a}^\ast\) we have 
\begin{equation}\label{b-estimate}
C^{-1}\prod_{j=1}^r(1+\mid \lambda_j\mid^2)^{\frac{b-\varepsilon(l)/2}{2}}\leq \mid \mathbf{b}(\lambda,l)\mid^{-1}\leq C \prod_{j=1}^r(1+\mid \lambda_j\mid^2)^{\frac{b-\varepsilon(l)/2}{2}},
\end{equation}
with \(\varepsilon(l)=\pm 1\) according to \(b+1\pm l\notin 2{\Z}^-\) or \(b+1\pm l\in 2{\Z}^-\).
\end{itemize}
\end{lemma}
\begin{proof}
\begin{itemize}
\item[(i)] 
Assume \(b+1\pm l\notin 2{\Z}^-\), then
\begin{equation*}
\mathbf{b}(\lambda,l)=\prod_{j=1}^{n}\dfrac{2^{-i\lambda_{j}}\Gamma(i\lambda_{j}+1)}{\Gamma(\dfrac{1}{2}(b+1+i\lambda_{j}+l))\Gamma(\dfrac{1}{2}(b+1+i\lambda_{j}-l))},
\end{equation*}
and clearly \(\mathbf{b}(\lambda,l)\) has no zero in \(\mathfrak{a}^\ast\).\\
When \(b+1\pm l\in 2{\Z}^-\), then \(\mathbf{b}(\lambda,l)\) has a priori zeros and poles in Weyl walls \[\displaystyle\cup_{j=1}^r \{\lambda\in \mathfrak{a}^\ast;\lambda_j=0\}.\]
We claim that this is not the case. Indeed,  suppose \(b+1-l\in 2{\Z}^-\), which implies  that \(l\geq b+1\).\\
We have  \(\displaystyle\lim_{\lambda_j\rightarrow 0}\dfrac{2^{-i\lambda_{j}}\Gamma(i\lambda_{j})}{\Gamma(\dfrac{1}{2}(b+1+i\lambda_{j}+l))\Gamma(\dfrac{1}{2}(b+1+i\lambda_{j}-l))}=\frac{1}{2}\frac{(-1)^{(l-b-1)/2}}{(\frac{l-b-1}{2})!\Gamma(\frac{l+b+1}{2})}\). Thus  \(\mathbf{b}(\lambda,l)\) is well defined on \(\mathfrak{a}^\ast\) and do not vanishes.
\item[(ii)] The proof uses the well known property of the \(\Gamma\) function
\begin{equation}\label{G-estimate}
\lim_{\mid z\mid \rightarrow+\infty} \frac{\Gamma(z+a)}{\Gamma(z)}{\rm e}^{-a\log z}=1, \, \, \mid arg z\mid\leq \pi-\delta,
\end{equation}
where \(a\) is any complex number, \(log\) is the principal value of the logarithm and \(\delta>0\).\\
Suppose first that \(b+1\pm l\notin 2{\Z}^-\). Put
\begin{equation*}
\mathbf{b}(\lambda_j,l)=\dfrac{2^{-i\lambda_{j}}\Gamma(i\lambda_{j}+1)}{\Gamma(\dfrac{1}{2}(b+1+i\lambda_{j}+l))\Gamma(\dfrac{1}{2}(b+1+i\lambda_{j}-l))}.
\end{equation*}
Then using the duplication formula for the gamma function
\begin{equation}\label{gamma}
\Gamma(2z)=\frac{2^{2z-2}}{\sqrt{\pi}}\Gamma(z)\Gamma(z+\frac{1}{2}),
\end{equation}
we may rewrite \(\mathbf{b}(\lambda_j,l)\) as 
\begin{equation*}
\mathbf{b}(\lambda_j,l)=\frac{1}{2\sqrt{\pi}}\dfrac{\Gamma(\frac{i\lambda_{j}+1}{2})\Gamma(\frac{i\lambda_{j}}{2}+1)}{\Gamma(\dfrac{1}{2}(b+1+i\lambda_{j}+l))\Gamma(\dfrac{1}{2}(b+1+i\lambda_{j}-l))},
\end{equation*}
and from (\ref{G-estimate}) we get 
\begin{equation*}
\mid \mathbf{b}(\lambda_j,l)\mid\leq C(1+\lambda_j^2)^{\frac{1-2b}{4}}
\end{equation*}
\begin{equation*}
\mid \mathbf{b}(\lambda_j,l)\mid^{-1}\leq C(1+\lambda_j^2)^{\frac{2b-1}{4}}.
\end{equation*}
Multiplying over \(j\) gives the result.\\
Next, if \(b+1-l\in 2{\Z}^-\), then \(\displaystyle\mathbf{b}(\lambda_j,l)=\dfrac{2^{-i\lambda_{j}}\Gamma(i\lambda_{j})}{\Gamma(\dfrac{1}{2}(b+1+i\lambda_{j}+l))\Gamma(\dfrac{1}{2}(b+1+i\lambda_{j}-l))}\).\\
We may use once again (\ref{gamma}) as well as (\ref{G-estimate}) to show as in i) that   \(\mathbf{b}(\lambda_j,l)\) and its inverse satisfy the following estimates
\begin{equation*}
\mid \mathbf{b}(\lambda_j,l)\mid\leq C(1+\lambda_j^2)^{\frac{-1-2b}{4}}
\end{equation*}
\begin{equation*}
\mid \mathbf{b}(\lambda_j,l)\mid^{-1}\leq C(1+\lambda_j^2)^{\frac{1+2b}{4}}.
\end{equation*}
\end{itemize}
this finishes the proof of the Lemma.
\end{proof}
For the proof of Proposition \ref{uniform}, we will need the following auxiliary result, see \cite{A}.\\ Let \(\eta\) be a positive Schwartz function on \({\R}\) whose Fourier transform has a compact support. For \(m\in {\R}\), set
\[\eta_{m}(x)=\int_{\R}\eta(t)(1+|t-x|^{2})^{m/2}\, {\rm d}t.\]
\begin{lemma}\cite{A}\label{Ank}
\item[(i)] \(\eta_m\) is a positive \(C^\infty\)-function with
\begin{equation}\label{A}
C^{-1}(1+|t|^{2})^{m} \leq \eta_{m}(t) \leq C(1+|t|^{2})^{m},
\end{equation}
for some positive constant \(C\).
\item[(ii)] The Fourier transform of \(\eta_m\) has a compact support.
\end{lemma}
Now we come to the proof of Proposition \ref{uniform}.
\begin{proof}
We first notice that it is sufficient to establish (\ref{F.uniform}) for functions \(F\in C_c^\infty(G,\tau_l)\) supported in the ball \(B(R)\). By the Plancherel formula (\ref{P}) we have
\[\int_{B(R)}\mid F(g)\mid^2 {\rm d}g_K\geq \dfrac{1}{|W|} \int_K\int_{\mathfrak{a}^\ast}\mid \widetilde{F}_l(\lambda,k)\mid^2\mid c(\lambda,l)\mid^{-2}{\rm d}\lambda {\rm d}k.
\]
Therefore it is sufficient  to prove
\begin{equation}\label{d1}
\int_K\int_{\mathfrak{a}^\ast}\mid \widetilde{F}_l(\lambda,k)\mid^2\mid c(\lambda,l)\mid^{-2}{\rm d}\lambda {\rm d}k\geq C R^{-r} \mid c(\lambda,l)\mid^{-2}\int_K\mid \widetilde{F}_l(\lambda,k)\mid^2{\rm d}k
\end{equation}
for some positive constant \(C\).\\
Define the function \(\widetilde{\mathbf{c}}(\lambda,l)^{-1}\) on \(\mathfrak{a}^\ast_{\textit{reg}}\) by
\begin{equation*}
\widetilde{\mathbf{c}}(\lambda,l)^{-1}=\pi_l(i\lambda)\prod_{j=1}^r \eta_{\frac{2b-\epsilon(l)}{4}}(\lambda_j).
\end{equation*}
Then from (\ref{A}) and(\ref{b-estimate}) we  get
\begin{equation}\label{c-estimate1}
C^{-1}\mid c(\lambda,l)\mid^{-1}\leq \mid \widetilde{\mathbf{c}}(\lambda,l)\mid^{-1}\leq C \mid c(\lambda,l)\mid^{-1}.
\end{equation}
Now (\ref{d1}) is equivalent to
\begin{equation}\label{d2}
\mid\widetilde{\mathbf{c}}(\lambda,l)\mid^{-2}\int_K\mid \widetilde{F}_l(\lambda,k)\mid^2{\rm d}k\leq  C R^r \int_K\int_{\mathfrak{a}^\ast}\mid \widetilde{F}_l(\lambda,k)\mid^2\mid \widetilde{\mathbf{c}}(\lambda,l)\mid^{-2}{\rm d}\lambda \,{\rm d}k,
\end{equation}
which is in turn  equivalent to
\begin{equation}\label{d3}
\int_K\mid \mathcal{F}_\mathfrak{a}(T\ast \mathcal{R}F(,k))(\lambda)\mid^2 {\rm d}k\leq  C R^r \int_K\int_{\mathfrak{a}^\ast}\mid \mathcal{F}_\mathfrak{a}(T\ast \mathcal{R}F(,k))(\lambda)\mid^2 {\rm d}\lambda\, {\rm d}k,
\end{equation}
where \(T\) is the tempered distribution defined  on \(\mathfrak{a}\) by \(\displaystyle T= \mathcal{F}_a^{-1}[\widetilde{\mathbf{c}}(,l)]^{-1}\). Moreover \(T\) has compact support in \(\mathfrak{a}\), by Lemma \ref{Ank}.  Let \(R_0>1\) such that \(\textit{supp}T\subset \overline{\mathfrak{a}(R_0)}\).\\
By using the euclidean Plancherel formula, we are left to prove
\begin{equation}\label{d4}
\int_K\mid \mathcal{F}_\mathfrak{a}(T\ast \mathcal{R}F(,k))(\lambda)\mid^2 {\rm d}k\leq  C R^r \int_K\int_{\mathfrak{a}^\ast}\mid (T\ast \mathcal{R}F(,k))(H)\mid^2 {\rm d}H\, {\rm d}k,
\end{equation}
where \(\ast\) denotes the convolution operator on \(\mathfrak{a}\).\\
Now, since \(F\) has compact support \(\displaystyle\subset \mathfrak{a}(R)\) then \(\displaystyle\textit{supp}\,\mathcal{R}F \subset \mathfrak{a}(R)\times K\) by lemma \ref{Radon}. Therefore
\[\textrm{supp}\, (T\ast \mathcal{R}F(.,k))\subset \overline{\mathfrak{a}(R+R_0)}, \quad k\in K.\]
Thus
\begin{equation}\begin{split}
\mid \mathcal{F}_\mathfrak{a}(T\ast \mathcal{R}F(,k))(\lambda)\mid^2\leq &\left(\int_{\mathfrak{a}(R+R_0)}{\rm d}H\right)\int_\mathfrak{a}\mid (T\ast \mathcal{R}F(,k))(H)\mid^2 {\rm d}H\\
&\leq C R^r\int_\mathfrak{a}\mid (T\ast \mathcal{R}F(,k))(H)\mid^2 {\rm d}H,
\end{split}\end{equation}
which gives (\ref{d4}), and the proof of Proposition \ref{uniform} is finished.
\end{proof}
We can now prove  the uniform continuity estimate for the Poisson transform \(P_{\lambda,l}\).
\begin{corollary}\label{cor1}
Let \(l\in {\Z}\). There exists a positive constant \(C\) such that for \(\lambda\in\mathfrak{a}_{\textit{reg}}^\ast\)
we have
\begin{equation}\label{P.uniform}
\sup_{R>1}\left(\frac{1}{R^r}\int_{B(R)}\mid P_{\lambda,l}f(g)\mid^2\, {\rm d}g_K\right)^{\frac{1}{2}}\leq C \mid c(\lambda,l)\mid \left\|f\right\|_2,
\end{equation}
for every \(f\in L^2(K,\tau_l)\).
\end{corollary}
\begin{proof}
Let \(F\in L^2(G,\tau_l)\) with \(supp F\subset B(R)\), then we have
\begin{equation*}
\int_{B(R)}P_{\lambda,l}f(g)\overline{F(g)}\, {\rm d}g_K=\int_K f(k)\overline{\widetilde F_l(\lambda,k)}\,{\rm d}k.
\end{equation*}
Now the estimate (\ref{F.uniform}) implies the result by standard duality argument.
\end{proof}
\section{Asymptotic formula for the Poisson transform}
In  this section we give an asymptotic expansion for the Poisson transform.
To do so we introduce the function space \(B_l^{*}(G)\) on G, consisting of functions \(F\) in \( L^2_{loc}(G,\tau_l)\) satisfying
\[
\|F\|_{B_l^{*}(G)}=\sup_{j \in \N}[2^{-(\frac{r}{2})j}\int_{D_j}\mid F(g)\mid^2\, {\rm d}g_K]<\infty,
\] 
where \(D_0=\{gK\in G/K; d(0,g.0)<1 \}, D_j=\{gK \in G/K; 2^{j-1}\leq d(0,g.0)< 2^{j}\}\) for  \(j\geq 1\).\\
We note for use below that the \(B_l^\ast(G)\) norm is equivalent to the norm \(\|F\|_*\). Indeed, since \(2^j\) is in geometric progression we have
\begin{align*}
\parallel F\parallel_{B_l^{*}(G)}^2\leq \sup_{R>1}\frac{1}{R^r}\int_{B(R)}\mid F(g)\mid^2\, {\rm d}g_K\leq 2^r(1-2^{-r})^{-1}\parallel F\parallel_{B_l^{*}(G)}^2
\end{align*}
Define an equivalent relation on \(B_l^{*}(G)\) as follows:\\
For \(F_{1}, F_{2} \in B_l^{*}(G)\) we write \( F_{1} \simeq F_{2}\) if
\[\lim_{R\rightarrow +\infty}\frac{1}{R^{r}}\int_{B(R)}|F_{1}(g)-F_{2}(g)|^{2}dg_K =0.\]
Then the main result of this section can be stated as follows
\begin{theorem}\label{asympt0}
Let \(l\in {\Z}\) and \(\lambda\in \mathfrak{a}^\ast_{\textit{reg}}\). For \(f\in L^2(K,\tau_l))\) we have the following asymptotic expansion for the Poisson transform in \(B_l^{*}(G)\).
\begin{align}\label{asympt5}
\hspace{1,5cm}P_{\lambda,l}f(g)\simeq \tau_l^{-1}(k_2(g))\sum_{s\in W} c(s\lambda,l){\rm e}^{(is\lambda-\rho)(A^+(g))}U^l_{s,\lambda} f(k_1(g)),
\end{align}
\(g=k_1(g)e^{A^{+}(g)}k_2(g)\).\\
Here \(U_{s,\lambda}\) are unitary operators on \(L^2(K,\tau_l))\) defined by \(U_{s,\lambda}=P_{s\lambda,l}^{-1}\circ P_{\lambda,l}\).
\end{theorem} 
Most part of the proof of Theorem \ref{asympt0} consists in proving the proposition below.\\
Denote by \(\displaystyle \mathfrak{a}^+(R)=\{H\in \mathfrak{a}^+; \mid H\mid<R\}\).
\begin{proposition}\label{asympt}
Let \(l\in {\Z}\), \( \lambda\in\mathfrak{a}^\ast_{\textrm reg}\). Then  we have the asymptotic expansion in \(B_l^{*}(G)\) for \(\varphi_{\lambda,l}\).
 \begin{equation}\label{asympt3}
\varphi_{\lambda,l}(g) \simeq  \tau_l^{-1}(\kappa_1(g)\kappa_2(g))\sum_{s\in W}c(s\lambda,l){\rm e}^{(is\lambda-\rho)(A^+(g))},
\end{equation}
\(g=k_1(g)e^{A^{+}(g)}k_2(g)\).
\end{proposition}
\begin{proof}
Let us first show that the right hand side of (\ref{asympt3}) lies in \(B_l^{*}(G)\).\\
Since \(\Delta({\rm e}^{H})\leq {\rm e}^{2\rho(H)}\) and \(\tau_l\) is unitary, we get
\[\frac{1}{R^r}\int_{B(R)}\mid \tau_l^{-1}(\kappa_1(g)\kappa_2(g))\sum_{s\in W}c(s\lambda,l){\rm e}^{(is\lambda-\rho)(A^+(g))}\mid^2\,{\rm d}g_K\leq \mid W\mid \mid c(\lambda,l)\mid^2\int_{ \mathfrak{a}^+(1)}d{\rm H}<+\infty,\]
as to be shown.\\We turn to the proof of the asymptotic formula for \(\varphi_{\lambda,l}\). We have
\begin{equation*}\begin{split}
&\frac{1}{R^r}\int_{B(R)}\mid \varphi_{\lambda,l}(g)-\tau_l^{-1}(\kappa_1(g)\kappa_2(g))\sum_{s\in W}c(s\lambda,l){\rm e}^{(is\lambda-\rho)(A^+(g))}\mid^2\, {\rm d}g_K\\
&\leq \frac{2}{R^r}\int_{\mathfrak{a}^+(R)}\mid \varphi_{\lambda,l}({\rm e}^H)-u({\rm e}^H)^{-l}\sum_{s\in W}c(s\lambda,l){\rm e}^{(is\lambda-\rho(l))(H)}\mid^2\,{\rm e}^{2\rho(H)} {\rm d}H\\
&+ \frac{2}{R^r}\int_{\mathfrak{a}^+(R)}\mid \sum_{s\in W}c(s\lambda,l)
 {\rm e}^{(is\lambda-\rho)(H)}[(u({\rm e}^{H}))^{-l}{\rm e}^{l\sum_{j=1}^r\alpha_j(H)}-1]\mid^2  \,{\rm e}^{2\rho(H)} {\rm d}H\\
&=: I_1(R)+I_2(R).
\end{split}\end{equation*} 
By using (\ref{estimate7}) and making the transformation \(H\rightarrow RH\) we see that \(\displaystyle I_1(R)\) is less than
\begin{equation*}
 2C_{\lambda}^{2}\int_{\mathfrak{a}^+(1)}{\rm e}^{-2R\tau(H)}\,{\rm d}H,
\end{equation*}
from which we deduce that \(\displaystyle\lim_{R\rightarrow +\infty} I_1(R)=0\).\\
For \(I_2(R)\), we have
\begin{equation*}\begin{split}
I_2(R)&\leq \frac{2}{R^r}\mid c(\lambda,l)\mid^{2}\mid W\mid \int_{\mathfrak{a}^+(R)}
\mid u({\rm e}^{H)})^{-l}{\rm e}^{l\sum_{j=1}^r\alpha_j(H)}-1\mid^2\, {\rm d}H\\
&\leq 2\mid c(\lambda,l)\mid^{2}\mid W\mid\int_{\mathfrak{a}^+(1)}\prod_{j=1}^r(1-(1+{\rm e}^{-2R\alpha_j(H)})^{-l})^2\, {\rm d}H.
\end{split}\end{equation*}
By the Lebesgue 's dominated convergence theorem, we easily see that \(\displaystyle\lim_{R\rightarrow +\infty}I_2(R)=0\).
Therefore
\[\lim_{R\rightarrow +\infty}\frac{1}{R^r}\int_{B(R)}\mid \varphi_{\lambda,l}(g)-\tau_l^{-1}(k_1(g)k_2(g))\sum_{s\in W}c(s\lambda,l){\rm e}^{(is\lambda-\rho)(A^+(g))}\mid^2\, {\rm d}g_K=0,\]
and the proof of proposition \ref{asympt} is finished.
 \end{proof}
\begin{lemma}
If \(\lambda\in \mathfrak{a}^\ast_{\textrm{reg}}\) then the functions
\begin{equation*}
k\rightarrow \sum_j a_j{\rm e}^{(i\lambda-\rho)H(g_j^{-1}k)}\tau_l^{-1}(\kappa(g_j^{-1}k)),\, a_j\in{\C}  \, g_j\in G.
\end{equation*}
form a dense subspace \(H_{\lambda,l}\) of \(L^2(K,\tau_l))\).
\end{lemma}
\begin{proof}
Noting that if \(\lambda\in \mathfrak{a}^\ast_{\textrm{reg}}\), then \(\lambda\) is simple. The density is just a reformulation of the definition of the simplicity.
\end{proof}
For \(g\in G\), put \(\displaystyle f_g^\lambda (k)={\rm e}^{(i\lambda-\rho)H(g^{-1}k)}\tau^{-1}_l(\kappa(g^{-1}k)), \quad k\in K\). Then
using the symmetric formula for the elementary \(\tau_{-l}\)-spherical function( see \cite{Ca})
\begin{equation}\label{S}
\varphi_{\lambda,l}(h^{-1}g)=\int_K  {\rm e}^{(i\lambda-\rho)H(h^{-1}k)} \tau_{-l}(\kappa(h^{-1}k)){\rm e}^{-(i\lambda+\rho)H(g^{-1}k)} \tau_l(\kappa(g^{-1}k))\,{\rm d}k,
\end{equation}
and the simplicity of \(s\lambda\) (\(s\in W\) and \(\lambda\in \mathfrak{a}^\ast_{\textrm{reg}}\)) we can prove the following result in a similar manner to [Lemma 6.7 \cite{Ka}].
\begin{lemma}
Let \(s\in W, l\in {\Z}\) and \(\lambda\in \mathfrak{a}^\ast_{\textrm{reg}}\). Then there exists a unitary operator \(U_{s,\lambda}^l\) on \(L^2(K,\tau_l)\) such that \[U_{s,\lambda}^l f_g^\lambda (k)=f_g^{s\lambda} (k).\]
Moreover, for \(f_1,f_2\in L^2(K,\tau_l)\) we have \(P_{\lambda,l}f_1=P_{s\lambda,l}f_2\) if and only if \(U_{s,\lambda}^l f_1=f_2\).
\end{lemma}
For the proof of Theorem \ref{asympt0} we shall need the following result.
\begin{lemma}
Let \(g,h\in G\) and put \(R(g,h)=A^+(h^{-1}g)-A^+(g)-H(h^{-1}\kappa_1(g))\).\\
(i) R(g,h) lies in the dual cone \(^{+}\overline{\mathfrak{a}}\) and there exists a positive constant \(C(h)\) bounded on each compact set in \(G\) such that
\begin{equation}\label{R}
\mid R(g,h)\mid\leq C(h) {\rm e}^{-2\tau(A^+(g))},
\end{equation}
for all \(g\in G\).\\
(ii) \begin{equation}\label{R1}
\mid {\rm e}^{(is\lambda-\rho)R(g,h)}-1\mid\leq C(\mid \lambda\mid+\mid\rho\mid)C(h){\rm e}^{-2\tau A^+(g)}.
\end{equation}
\end{lemma}
\begin{proof}
\begin{itemize}
\item[(i)] The proof of (\ref{R}) is essentially the same as [\cite{Ka},Lemma 5.2], so we omit it (see also [\cite{He}, chapter II,Proposition 4.24].
\item[(ii)] the estimate (\ref{R1}) follows from the  mean value theorem and inequality (\ref{R}) .
\end{itemize}
\end{proof} 
\textbf{Proof of Theorem \ref{asympt0}}
We shall first take the case when \(f=f^\lambda_h\), \(h\) being fixed in \(G\).\\
Noting that \(P_{\lambda,l}f^\lambda_h(g)=\varphi_{\lambda,l}(h^{-1}g)\) and \(U_{s,\lambda}f^\lambda_h(\kappa_1(g))= {\rm e}^{(is\lambda-\rho)H(h^{-1}\kappa_1(g))}\tau_l^{-1}\kappa(h^{-1}\kappa_1(g))\), it is enough to show that 
\begin{equation}\begin{split}\label{asymp4}
&\tau_l^{-1}(\kappa_1(h^{-1}g)\kappa_2(h^{-1}g))\sum_{s\in W}c(s\lambda,l){\rm e}^{(is\lambda-\rho)A^+(h^{-1}g)}\\
&\simeq \tau_l^{-1}(\kappa_2(g))\sum_{s\in W}c(s\lambda,l){\rm e}^{(is\lambda-\rho)[A^+(g)+H(h^{-1}\kappa_1(g)]}\tau_l^{-1}(\kappa(h^{-1}\kappa_1(g)), 
\end{split}\end{equation} 
by (\ref{asympt3}).\\
We replace \(A^+(h^{-1}g)\) by \(A^+(g)+H(h^{-1}\kappa_1(g))+R(g,h)\) and decompose the left hand side of (\ref{asymp4}) as
\begin{equation*}
\tau_l^{-1}(\kappa_2(g))\tau_l^{-1}(\kappa(h^{-1}\kappa_1(g))\sum_{s\in W}c(s\lambda,l){\rm e}^{(is\lambda-\rho)[A^+(g)+H(h^{-1}\kappa_1(g)]}+J(g),
\end{equation*}
where \begin{equation*}\begin{split}
&J(g)=\tau_l^{-1}(\kappa_1(h^{-1}g)\kappa_2(h^{-1}g))\sum_{s\in W}c(s\lambda,l){\rm e}^{(is\lambda-\rho)[A^+(g)+H(h^{-1}\kappa_1(g)]}[{\rm e}^{(is\lambda-\rho)R(g,h)}-1]\\
&+[\tau_l^{-1}(\kappa_1(h^{-1}g)\kappa_2(h^{-1}g))-\tau_l^{-1}(\kappa(h^{-1}\kappa_1(g)))\tau_l^{-1}(\kappa_2(g))]\sum_{s\in W}c(s\lambda,l){\rm e}^{(is\lambda-\rho)[A^+(g)+H(h^{-1}\kappa_1(g)]}.
\end{split}\end{equation*}
Thus we are left to show that \(\displaystyle\lim_{R\rightarrow +\infty}\frac{1}{R^r}\int_{B(R)}\mid J(g)\mid^2 {\rm d}g_K=0\).\\
Let us denote by \(J_1(g)\) (respectively \(J_2(g)\)) the first (respectively the second) term in the above sum.\\
Put \(C_h(\lambda)=(C(\mid \lambda\mid+\rho\mid)C(h))^2\mid W\mid\mid c(\lambda,l)\mid^2\). Then using (\ref{R1})  and noting that \(\int_K {\rm e}^{-2\rho(H(h^{-1}k)}\, {\rm d}k=1\), we get
\begin{equation*}
\frac{1}{R^r}\int_{B(R)}\mid J_1(R)\mid^2 {\rm d}g_K\leq C_h(\lambda)\frac{1}{R^r}\int_{\mathfrak{a}^+(R)}{\rm e}^{-4\tau(H)}{\rm d}H,
\end{equation*}
from which we deduce that \(\displaystyle\lim_{R\rightarrow +\infty} \frac{1}{R^r}\int_{B(R)}\mid J_1(R)\mid^2 {\rm d}g_K =0\).\\
To end the proof we will need the following result, the proof of which will be done in the Appendix.
\begin{lemma A2}\label{decomposition}
Let \(\displaystyle h,g\in G\). Then we have
\begin{itemize}
\item[(i)] \(\displaystyle\kappa_1(h^{-1}g)\kappa_2(h^{-1}g)=\kappa_1(h^{-1}\kappa_1(g){\rm e}^{A^+(g)})\kappa_2(h^{-1}\kappa_1(g){\rm e}^{A^+(g)})\kappa_2(g)\).
\item[(ii)] For $H\in \mathfrak{a}^{+}$,\(\displaystyle\lim_{R\rightarrow +\infty}\tau_l(\kappa_1(g{\rm e}^{RH})\kappa_2(g{\rm e}^{RH}))=\tau_l(\kappa(g))\).
\end{itemize}
\end{lemma A2}
Using (i) in Lemma \ref{decomposition} and making the transform \(\displaystyle R\rightarrow RH\) we easily see  that\\
\(\displaystyle\frac{1}{R^r}\int_{B(R)}\mid J_2(R)\mid^2 {\rm d}g_K\) is majorized by \begin{equation*}
\mid W\mid \mid c(\lambda,l)\mid^2\int_{K\times \mathfrak{a}^+(1)}\mid \tau_l^{-1}(\kappa_1(h^{-1}k{\rm e}^{RH})\kappa_2(h^{-1}k{\rm e}^{RH}))-\tau_l^{-1}(\kappa(h^{-1}k))\mid^2{\rm e}^{-2\rho H(h^{-1}k)}{\rm d}k{\rm d}H.
\end{equation*}
which in turn is less than
\begin{equation*}
\mid W\mid \mid c(\lambda,l)\mid^2 \sup_{k\in K}{\rm e}^{-2\rho H(h^{-1}k)}\int_{K\times \mathfrak{a}^+(1)}\mid \tau_l^{-1}(\kappa_1(h^{-1}k{\rm e}^{RH})\kappa_2(h^{-1}k{\rm e}^{RH}))-\tau_l^{-1}(\kappa(h^{-1}k))\mid^2{\rm d}k{\rm d}H.
\end{equation*}
It follows from (ii) in Lemma \ref{decomposition} that
\[\lim_{R\rightarrow +\infty}\int_{K\times\mathfrak{a}^+(1)}\mid \tau_l^{-1}(\kappa_1(h^{-1}k{\rm e}^{RH})\kappa_2(h^{-1}k{\rm e}^{RH}))-\tau_l^{-1}(\kappa(h^{-1}k))\mid^2{\rm d}H=0,
\]
by the Lebesgue's dominated convergence theorem. \\Thus \(\displaystyle\lim_{R\rightarrow +\infty}\frac{1}{R^r}\int_{B(R)}\mid J_2(R)\mid^2 {\rm d}g_K=0\) and therefore \(\displaystyle\lim_{R\rightarrow +\infty}\frac{1}{R^r}\int_{B(R)}\mid J(g)\mid^2 {\rm d}g_K=0\) as to be shown.\\
 Now we prove that (\ref{asympt5}) holds for \(f\in L^2(K,\tau_l)\). First of all we introduce the operator \(S_\lambda\)
defined for\(f\in L^2(K,\tau_l)\) by
\begin{equation*}
S_{\lambda,l} f(g)=\tau_l^{-1}(k^\prime)\sum_{s\in W} \mathbf{c}(s\lambda,l){\rm e}^{(is\lambda-\rho)(H)}U^l_{s,\lambda} f(k), \quad  g=k{\rm e}^H k^\prime.
\end{equation*}
 For any \(\varepsilon>0\) there exists \(\phi\in H_{\lambda,l}\) such that \(\parallel f-\phi\parallel_2\leq \epsilon\). We have
\begin{equation*}\begin{split}
\frac{1}{R^r}\int_{B(R)}\mid P_{\lambda,l}f(g)-S_{\lambda,l}f(g)\mid^2{\rm d}g_K\leq 3 &(\frac{1}{R^r}\int_{B(R)}\mid P_{\lambda,l}(f-\phi)(g)\mid^2{\rm d}g_K\\
&+\frac{1}{R^r}\int_{B(R)}\mid (P_{\lambda,l}\phi-S_{\lambda,l}\phi)(g)\mid^2{\rm d}g_K\\
&+\frac{1}{R^r}\int_{B(R)}\mid S_{\lambda,l}(f-\phi)(g)\mid^2{\rm d}g_K).
\end{split}\end{equation*}
From (\ref{P.uniform}) we see that the first term of the right hand side of the above inequality is less than
\(C\mid c(\lambda,l)\mid^2\varepsilon^2\).
Also, the (i) part  implies
\begin{equation*}
\lim_{R\rightarrow +\infty}\frac{1}{R^r}\int_{B(R)}\mid (P_{\lambda,l}\phi-S_{\lambda,l}\phi)(g)\mid^2{\rm d}g_K=0.
\end{equation*}
Next, by the unitarity of \(U_{s,\lambda}^l\) we easily see that 
\begin{equation*}\begin{split}
\frac{1}{R^r}\int_{B(R)}\mid S_{\lambda,l}(f-\phi)(g)\mid^2{\rm d}g_K &\leq \mid c(\lambda,l)\mid^2 \mid W\mid\sum_{s\in W}\frac{1}{R^r}\int_{K\times\mathfrak{a}^+(R)} | U_{s,\lambda}^l(f-\phi)(k)|^2 {\rm d}H{\rm d}k\\
&\leq \mid c(\lambda,l)\mid^2 \mid W\mid^2\varepsilon^2.
\end{split}\end{equation*}
Therefore
\begin{align*}
\lim_{R\rightarrow +\infty}\frac{1}{R^r}\int_{B(R)}\mid S_{\lambda,l}(f-\phi)(g)\mid^2{\rm d}g_K\leq \mid c(\lambda,l)\mid^2 \mid W\mid^2\varepsilon^2,
\end{align*}
which implies that (\ref{asympt5}) holds for any \(f\in L^2(K,\tau_l)\) and the proof of Theorem \ref{asympt0} is completed.
\section{The \(L^2\)-range of the Poisson transform}
As preparation to the proof of Theorem\ref{main R}  we prove the following:
\begin{proposition}\label{Poisson norm}
Let \(l\in{\Z}\) and  \(\lambda\in \mathfrak{a}^\ast_{\textit{reg}}\). For \(f\in L^2(K,\tau_l)\) we have
\begin{equation}\label{N}
\lim_{R\rightarrow +\infty}\frac{1}{R^r}\int_{B(R)}\mid P_{\lambda,l}f(g)\mid^2\,{\rm d}g_K=2^{-r/2}\Gamma(r/2+1)^{-1}\mid c(\lambda,l)\mid^2  \left\|f\right\|^2_2.
\end{equation}
\end{proposition}
\begin{proof}
Let \(f\in L^2(K,\tau_l)\). We have
\begin{eqnarray*}
\dfrac{1}{R^{r}}\int_{B(R)}|\mathcal{P}_{\lambda,l}f(g)|^{2} {\rm d}g_K&=&\dfrac{1}{R^{r}}\int_{B(R)}|\mathcal{P}_{\lambda,l}f(g)-S_{\lambda,l}f(g)|^{2}{\rm d}g_K+\dfrac{1}{R^{r}}\int_{B(R)}|S_{\lambda,l}f(g)|^2{\rm d}g_K\\
&+&\dfrac{2}{R^{r}}\int_{ B(R)}Re\left((\mathcal{P}_{\lambda,l}f(g)-S_{\lambda,l}f(g))\overline{S_{\lambda,l}f(g)} \right){\rm d}g_K \\
\end{eqnarray*}
The first term in the right hand side of the above identity goes to \(0\) as \(R\rightarrow +\infty\), by  Theorem \ref{asympt0}.\\
By using the Schwartz inequality we easily see that the third term goes to \(0\) as  \(R\) goes to \( +\infty\), since 
\(\displaystyle\sup_{R>1}\dfrac{1}{R^{r}}\int_{B(R)}|S_{\lambda,l}(f)(g)|^{2}dg_K<+\infty \).\\
For the second term, by the unitarity of \(U_{s,\lambda}\) we have
\begin{equation*}\begin{split}
&\dfrac{1}{R^{r}}\int_{B(R)}|S_{\lambda,l}f(g)|^{2}{\rm d}g_K=|c(\lambda,l)|^{2}\left\|f\right\|^2_2 |W|\dfrac{1}{R^{r}}\int_{\mathfrak{a}^+(R)}e^{-2\rho(H)}\Delta(e^{H})dH+\\
&+\sum_{s_{1},s_{2} \in W; s_1\neq s_2}c(s_{1}\lambda,l)\overline{c(s_{2}\lambda,l)}\left(U^{l}_{s_{1},\lambda}f,U^{l}_{s_{2},\lambda}f\right)_{L^{2}(K)}\times J(R),
\end{split}\end{equation*}
where
\begin{equation*}
J(R)= \int_{\mathfrak{a}^+(1)}e^{iR(s_{1}\lambda-s_{2}\lambda)(H)}\Delta(e^{RH})e^{-2R\rho(H)}dH.
\end{equation*}
To finish the proof, it is enough  to show that \[\lim_{R\rightarrow +\infty}J(R)=0 \quad \textit{and} \quad \lim_{R\rightarrow +\infty} \dfrac{1}{R^{r}}|W|\int_{\mathfrak{a}(R)}e^{-2\rho(H)}\Delta(e^{H})dH=2^{-r/2
}\Gamma(r/2+1)^{-1}.\]
We write \(J(R)=J_1(R)+J_2(R)\), with
\begin{eqnarray*}
J_1(R)=\int_{\mathfrak{a}^+(1)}e^{iR(s_{1}-s_{2})\lambda(H)}dH,
\end{eqnarray*}
and
\begin{eqnarray*}J_2(R)=\int_{\mathfrak{a}^+(1)}e^{iR(s_{1}\lambda-s_{2}\lambda)(H)}(\Delta(e^{RH})e^{-2R\rho(H)}-1)dH.
\end{eqnarray*}
Since \(s_1\lambda\neq s_2\lambda\) because \(s_1\neq s_2\), we have \(\lim_{R\rightarrow +\infty}J_1(R)=0\), by the Riemann-Lebesgue lemma.\\
We also have
\begin{equation*}\begin{split}
 \int_{\mathfrak{a}^+(1)}|\Delta(e^{RH})e^{-2R\rho(H)}-1|dH
 &\leq\int_{\mathfrak{a}^+(1)}|\prod_{\alpha\in \Sigma^{+}}(1+e^{-2R\alpha(H)})^{m_{\alpha}}-1|dH,
\end{split}\end{equation*}
which implies that  \[\displaystyle\lim_{R\rightarrow +\infty}J_{2}(R)=0\]
and 
\[\lim_{R\rightarrow +\infty} \dfrac{1}{R^{r}}|W|\int_{\mathfrak{a}^+(R)}e^{-2\rho(H)}\Delta(e^{H})dH=2^{-r/2
}\Gamma(r/2+1)^{-1}.\]
Therefore  \(\displaystyle\lim_{R\rightarrow +\infty}J(R)=0\) as to be shown and the proof is completed.
\end{proof} 
Let \(\widehat{K}\) be the set of  equivalence classes of  irreducible unitary representations of \(K\). For each \(\delta\in \widehat{K}\) let \(V_\delta\) be a representation space for \(\delta\). Set
\begin{equation*}
V_\delta^l=\{v\in V_\delta; \delta(m)v=\tau_l(m)v\quad  \textrm{for all}\quad  m\in M \},
\end{equation*}
and \(\widehat{K}_l=\{\delta\in \widehat{K};  V^l_\delta\neq 0\}\).\\ We choose an orthonormal base \(\{ v_1,...,v_{d(\delta)} \}\) of \(V_\delta\) with respect to the unitary inner product \(<,>\) of \(V_\delta\) so that the first \(m(\delta)\)-vectors form a basis of \(V_\delta^l\), where \(d(\delta)\) is the dimension of \(V_\delta\).\\
Let \(L^2_\delta(K,\tau_l)\) denote the subspace of \(L^2(K,\tau_l))\) consisting of the \(K\)-finite \\ functions which are of type \(\delta\). Then the functions
\begin{equation*}
f^\delta_{ij}(k)=\sqrt{d(\delta)}<v_j,\delta(k)v_i>,\quad 1\leq j\leq d(\delta),\quad  1\leq i\leq m(\delta),
\end{equation*}
form an orthonormal base of \(L^2_\delta(K,\tau_{l_{\mid M}})\) and
\[
\{f^\delta_{ij}; \delta\in \widehat{K}_l, \quad 1\leq j\leq d(\delta),\quad  1\leq i\leq m(\delta)\}
\]
is an orthonormal base of \(L^2(K,\tau_l))\).\\
Let \(f\in L^2(K,\tau_l)\). Define the operator \(A_\delta\in Hom(V_\delta,V_\delta^l)\) by  \(A_\delta=\int_K f(k)\delta(k)dk\). Then \(f\) has the Fourier series expansion 
\[f(k)=\sum_{\delta\in \widehat{K}_l}d(\delta) Tr(A_\delta\delta(k^{-1})),\]
and we have the Plancherel formula
\[ \left\|f\right\|_2^2 = \sum_{\delta \in \widehat{K}_l}d(\delta)\parallel A_\delta\parallel^2_{\textit{HS}},\]
where \(\parallel \parallel_{HS}\) stands for the Hilbert-Schmidt norm.\\
Consider the Eisenstein integral
\begin{equation*}
\Phi_{\lambda,\delta}^l(g)=\int_K{\rm e}^{(i\lambda-\rho)H(g^{-1}k)}\tau_{-l}(\kappa(g^{-1}k))\delta(k){\rm d}k.
\end{equation*}
Then \(\Phi_{\lambda,\delta}^l\) maps \(G\) into \( {\rm Hom}(V_\delta,V_\delta)\) and satisfies
\begin{equation*}
\Phi_{\lambda,\delta}^l(kgh)=\delta(k)\Phi_{\lambda,\delta}^l(g)\tau_l(h),
\end{equation*}
for all  \(k,h\in K, g\in G\).
\begin{lemma}
Let  \(\lambda \in \mathfrak{a}^{*}_{reg}\) ,\(\delta \in \widehat{K_l}\).\\
(i) For \(v_1,v_2\in V_\delta\) and \(w_1,w_2 \in V_{\delta}^{l}\), we have
\begin{equation}\label{asymp1}
\lim_{R\rightarrow +\infty}\dfrac{1}{R^{r}}\int_{B(R)}<v_1,\Phi_{\lambda,\delta}(g)w_1>\overline{<v_1,\Phi_{\lambda,\delta}(g)w_1>}\,{\rm d}g_K=C_r^2|c(\lambda,l)|^{2}\dfrac{<v_1,v_2><w_1,w_2>}{d(\delta)}.
\end{equation}
where \(C_r=2^{-r/4}\Gamma(r/2+1)^{-1/2}\).\\
(ii) Define  \(\displaystyle I_R(\lambda,\delta)\in End(V_\delta)\) by 
\begin{align*}
I_R(\lambda,\delta)=\frac{C_r^{-2}|c(\lambda,l)|^{-2}}{R^r}\int_{B(R)}\Phi_{\lambda,\delta}^l(g)^{*}\Phi_{\lambda,\delta}^l(g){\rm d}g_K.
\end{align*}
Then we have 
\begin{equation}\label{asymp2}
\lim_{R\rightarrow +\infty} I_R(\lambda,\delta)=Id_{V_\delta^l}, \quad \textit{on} \quad V_\delta^l.
\end{equation}
Moreover there exists a positive constant \(C\) such that for \(\lambda \in \mathfrak{a}^{*}_{reg}, \delta \in \widehat{K_l}\) and \(R>1\) we have
\[\parallel I_R(\lambda,\delta)\parallel\leq C,\]
where \(\parallel \, \parallel\) stands for the operatorial norm.
\end{lemma}
\begin{proof}
(i) Noting that \(<v_i,\Phi_{\lambda,\delta}(g)w_i>= P_{\lambda,l}(<v_i,\delta()>)(g) (i=1,2)\), it follows from Proposition \ref{Poisson norm} and Schur  orthogonality  relations that 
\begin{equation*}\begin{split}
\lim_{R\rightarrow +\infty}\dfrac{1}{R^{r}}&\int_{B(R)}<v_1,\Phi_{\lambda,\delta}(g)w_1>\overline{<v_1,\Phi_{\lambda,\delta}(g)w_1>}\,{\rm d}g_K\\
&=C_r|c(\lambda,l)|^{2}\int_K <v_1,\delta(k)w_1>\overline{<v_2,\delta(k)w_2>}\, {\rm d}k\\
&=C_r|c(\lambda,l)|^{2}\frac{<v_1,v_2>\overline{<w_1,w_2>}}{d_\delta}
\end{split}\end{equation*}
(ii) Let \(v,w\in V_\delta^l\). It follows from (\ref{asymp1}) that 
\begin{equation*}\begin{split}
\lim_{R\rightarrow +\infty}<I_R(\lambda,\delta)u,v>&=\lim_{R\rightarrow +\infty}\frac{C_r^{-2}|c(\lambda,l)|^{-2}}{R^r}\sum_{j=1}^{d_\delta}\int_{B(R)}<v_j,\Phi_{\lambda,\delta}(g)w>\overline{<v_j,\Phi_{\lambda,\delta}(g)v>}\, {\rm d}g_K\\
&=\sum_{j=1}^{d_\delta}\frac{\overline{<w,v>}}{d_\delta}\\
&=<v,w>,
\end{split}\end{equation*}
and the result follows.
\end{proof}
We note that (\ref{asymp2}) implies 
\begin{align}\label{asympt1}
\lim_{R\rightarrow +\infty}\frac{C_r^{-2}|c(\lambda,l)|^{-2}}{R^r}\int_{B(R)}\Phi_{\lambda,\delta}^l(a)^{*}\Phi_{\lambda,\delta}^l(a)\Delta(a){\rm d}a=Id_{V_\delta^l}, \quad \textit{on} \quad V_\delta^l.
\end{align}
\textbf{Proof of theorem \ref{main R}}
\begin{proof}
(i) The right hand side of the estimates (\ref{estimate0}) has already been proved, see \\corollary \ref{cor1}. On the other hand the inequality on the left hand side of (\ref{estimate0}) is a direct consequence of (\ref{N}).\\
To prove the assertion (ii), we have only to show that \(P_{\lambda,l}\)  maps \(L^2(K,\tau_l))\) onto \(\mathcal{E}^2_{\lambda,l}(G)\).\\
Let \(\displaystyle F\in \mathcal{E}^2_{\lambda,l}(G)\). By Theorem \ref{Shimeno}, we know that \(F=P_{\lambda,l}f\) for some \(f\in B(K,\tau_l)\). Let \(\displaystyle f(k)=\sum_{\delta \in \widehat{K}_l}d(\delta) Tr(A_\delta\delta(k^{-1}))\) be its Fourier series. Then
\[F(g)=\sum_{\delta\in \widehat{K}_l}d(\delta)Tr(A_\delta\Phi_{\lambda,\delta}^l(g)^\ast).\]
From the assumption \(\parallel F\parallel_\ast<\infty\) we have 
\begin{equation*}\begin{split}
\infty>\parallel F\parallel^2_\ast&=\sup_{R>1} \frac{1}{R^r}\int_{B(R)}\mid F(g)\mid^2\, {\rm d}g_K\\
&=\sup_{R>1}\frac{1}{R^r}\int_{K\times\mathfrak{a}^+(R)}\mid F(ka)\mid^2\Delta(a)\, {\rm d}k\,{\rm d}a\\
&=\sup_{R>1}\sum_{\delta \in \widehat{K}_l}d(\delta)\frac{1}{R^r}\int_{\mathfrak{a}^+(R)}\parallel A_\delta\Phi_{\lambda,\delta}(a)^\ast\parallel_{HS}^2\Delta(a)\,{\rm d}a.
\end{split}\end{equation*}
Then, if \(\Lambda\) is an arbitrary finite subset of \(\widehat{K}_l\), we get
\begin{align*}
\sum_{\delta \in \Lambda}d(\delta)\frac{1}{R^r}\int_{\mathfrak{a}^+(R)}\parallel A_\delta\Phi_{\lambda,\delta}(a)^\ast\parallel_{HS}^2\Delta(a)\,{\rm d}a\leq \parallel F\parallel^2_\ast,
\end{align*}
for every \(R>1\) and from (\ref{asympt1}) it follows that 
\begin{align*}
C_r^2|c(\lambda,l)|^2\sum_{\delta \in \Lambda}d(\delta)\parallel A_\delta\parallel_{HS}^2\leq \parallel F\parallel^2_\ast, 
\end{align*}
which implies that \(\displaystyle f\in L^2(K,\tau_l)\) and that \(\displaystyle C_r^2|c(\lambda,l)|^2\parallel f\parallel^{2}_2\leq \parallel F\parallel^2_\ast\).\\
(iii) Now we turn to the proof of the \(L^2\)-inversion formula.\\
 Let \(F\in \mathcal{E}^2_{\lambda,l}(G)\). By the second part of Theorem \ref{main R}, we know that there exists a unique \(f\) in \(L^{2}(K,\tau_{l})\) such that  \(F=\mathcal{P}_{\lambda,l}f\). Hence, if  \( f(k)=\sum_{\delta \in \widehat{K}_l}d(\delta)Tr(A_\delta\delta(k^{-1})\) is its Fourier series expansion, with \(\displaystyle A_\delta=\int_K f(k)\delta(k)\, {\rm d}k\), then
\begin{align}\label{Fourier}
F(ka)=\sum_{\delta \in \widehat{K}_l}d(\delta)Tr(A_\delta\Phi_{\lambda,\delta}^\ast(a))\delta(k^{-1})
\end{align}
For \(R>1\) we define the \({\C}\)-valued function \(f_R\) on \(K\) by 
\[f_{R}(k)=\dfrac{C_r^{-2}\mid c(\lambda,l)\mid^{-2}}{R^{r}}\int_{B(R)}e^{(i\lambda-\rho)H(g^{-1}k)}\tau_{-l}(\kappa(g^{-1}k))F(g)dg_K.\]
To prove that \(\displaystyle f_R \rightarrow f\) on \(L^2(K,\tau_l)\) as \(R\rightarrow \infty\), it is enough to establish 
\( \displaystyle \lim_{R\rightarrow +\infty}\left\|f_{R}\right\|_2=\left\|f\right\|_2 \), since 
\[\displaystyle <f,f>=C_{r}^{-2}|c(\lambda,l)|^{-2}\lim_{R\rightarrow +\infty}<f_{R},f>,\] by  (\ref{N}).\\
Let \(\displaystyle f_R(k)=\sum_{\delta \in \widehat{K}_l}d(\delta)tr(A_{\delta,R}\delta(k^{-1}))\) is the Fourier series expansion of \(f_R\) in \(L^2(K,\tau_l)\), with \(\displaystyle A_{\delta,R}=\int_{K}f_{R}(k)\delta(k)dk\).
It follows from (\ref{Fourier}) that 
\begin{equation*}\begin{split}
A_{\delta,R}&=\frac{C_{r}^{-2}|c(\lambda,l)|^{-2}}{R^r}\int_{B(R)}F(g)\Phi_{\lambda,\delta}(g)\, {\rm d}g_K\\
&=\frac{C_{r}^{-2}|c(\lambda,l)|^{-2}}{R^r}\int_{\mathfrak{a}^+(R)\times K}F(ka)\delta(k)\Phi_{\lambda,\delta}(a)\Delta(a)\,{\rm d}k\,{\rm d}a\\
&=\frac{C_{r}^{-2}|c(\lambda,l)|^{-2}}{R^r}A_\delta\int_{\mathfrak{a}^+(R)}\Phi_{\lambda,\delta}(a)^\ast\Phi_{\lambda,\delta}(a)\Delta(a)\,{\rm d}a
\end{split}\end{equation*}
using once again (\ref{asymp1}) we see that \[ A_{\delta,R}\rightarrow A_\delta \quad \textit{on}\quad  V_\delta^l, \quad as \quad R\rightarrow \infty.\]
Next, since  \(\displaystyle \parallel A_{\delta,R}\parallel_{HS} \leq C \parallel A_\delta\parallel_{HS}\) and noting that the series \(\sum_{\delta\in \widehat{K}_l}\parallel A_\delta\parallel_{HS}\) is convergent we conclude that \[ \lim_{R\rightarrow \infty}\parallel f_R\parallel_2=\parallel f\parallel_2.\] This ends the proof of Theorem \ref{main R}.
\end{proof}
\section{APPENDIX}
We first review some facts on  Jacobi functions. We refer to \cite{Kr} for more details.\\
For \(\alpha,\beta, \lambda \in {\C}\), \(\alpha \neq -1-2,\cdots\), the Jacobi functions of the first kind \(\displaystyle\phi^{\alpha,\beta}_\lambda(t)\) is the solution of the differential equation 
\begin{equation}\label{Jacobi equation}
\frac{d^2}{dt^2}\phi_\lambda^{\alpha,\beta}+[2(\alpha+1)\coth t+(2\beta+1)\tanh t]\frac{d}{dt}\,\phi_\lambda^{\alpha,\beta}=-(\lambda^2+(\alpha+\beta+1)^2)\phi_\lambda^{\alpha,\beta},
\end{equation}
which satisfied \(\phi_\lambda^{\alpha,\beta}(0)=1 \) and \(\frac{d}{dt}\phi_\lambda^{\alpha,\beta}(t)_{\mid t=0} =0 \).
The Jacobi functions can be expressed by the Gauss Hypergeometric functions as 
\begin{align*}
\phi^{(\alpha,\beta)}_\lambda(t)=F(\frac{i\lambda+\alpha+\beta}{2},\frac{-i\lambda+\alpha+\beta}{2}; \alpha+1; -\sinh^2 t).
\end{align*}
For \(\lambda\notin -i{\N}\), another solution \(\Phi_\mu^{\alpha,\beta}\) of (\ref{Jacobi equation}) such that
\begin{equation}\label{Jacobi asymp}
\Phi_\mu^{\alpha,\beta}(t)={\rm e}^{(i\mu-\alpha-\beta-1)t}(1+\circ(1)) \quad \textit{as} \,\, t\rightarrow +\infty,
\end{equation}
is given by
\begin{equation*}\label{Jacobi2}
\Phi_\lambda^{\alpha,\beta}(t)=(2\sinh t)^{i\lambda-\alpha-\beta-1}F(\frac{-i\lambda+\alpha+\beta+1}{2},\frac{-i\lambda-\alpha+\beta+1}{2}, 1-i\lambda; -\sinh^{-2} t).
\end{equation*} 
For \(\lambda\notin {\Z}\), \(\Phi_\lambda^{\alpha,\beta}\) and \(\Phi_{-\lambda}^{\alpha,\beta} \) are linearly independent solutions of (\ref{Jacobi equation}), so we have
\begin{equation*}\label{Jacobi3}
\phi_\lambda^{\alpha,\beta}(t)=c_{\alpha,\beta}(\lambda)\Phi_\lambda^{\alpha,\beta}(t)+c_{\alpha,\beta}(-\lambda)\Phi_{-\lambda}^{\alpha,\beta}(t),
\end{equation*}
where
\begin{equation*}\label{J4}
c_{\alpha,\beta}(\lambda)=\frac{2^{\alpha+\beta+1-i\lambda}\Gamma(\alpha+1)\Gamma(i\lambda)}{\Gamma(\frac{\alpha+\beta+1+i\lambda}{2})\Gamma(\frac{\alpha-\beta+1+i\lambda}{2})}.
\end{equation*}
We shall also need some estimates on the Jacobi functions from \cite{Kr}, but which will be stated here for \(\alpha,\beta, \lambda\in {\R}, \alpha>-\frac{1}{2}\).
\begin{lemma}\cite{Kr}
Assume \(\alpha,\beta\) with \(\alpha>-\frac{1}{2}\).
 \enumerate
\item[(i)] For each  and \(\delta>0\) there exists a positive constant \(C\) such that for all \(t\geq \delta\) and all \(\lambda\in{\R}\)
\begin{align}\label{A1}
\mid \Phi_{\lambda}^{\alpha,\beta}(t)\mid\leq C{\rm e}^{-(\alpha+\beta+1)t}.
\end{align}
\item[(ii)] There exists a positive constant \(C\) such that for all \(t\geq 0\) and all \(\lambda\in{\R}\) 
\begin{align}\label{A2}
\mid \frac{d^n}{dt^n}\phi_\lambda^{\alpha,\beta}(t)\mid\leq C(1+\mid \lambda\mid)^n(1+t){\rm e}^{-(\alpha+\beta+1)t}.
\end{align}
\end{lemma}  
\begin{lemma A1}
For \(\alpha ,\beta \in \R \) with \(\alpha>-\frac{1}{2}\)  and \(n\in \N\) there exists a positive constant \(C\) such that for all \(t\geq 0\) and \(\lambda\in {\R}\setminus\{0\}\)
\[\mid \lambda\dfrac{d^{n}}{dt^{n}}\phi_{\lambda}^{\alpha,\beta}(t)\mid\leq C(1+\lambda^{2})^{n+1}e^{-(\alpha+\beta+1) t}\]
\end{lemma A1}
\begin{proof}
Consider first the case \(n=0\). Since \(\alpha>-\frac{1}{2}\) it follows from the Stirling formula that there exists a positive constant \(C\) such that for all \(\lambda\in {\R}\setminus\{0\}\) \(\displaystyle \mid \lambda c(\lambda)\mid\leq C(1+\mid \lambda\mid)\).\\
It follows from (\ref{A1}) that there exists \(C>0\) such that for all \(t\geq 1\)
\begin{eqnarray*}
\mid \lambda\phi_\lambda^{\alpha,\beta}(t)\mid &\leq &C(1+\mid\lambda\mid){\rm e}^{-(\alpha+\beta+1)t}\\
 &\leq&2C(1+\lambda^{2}){\rm e}^{-(\alpha+\beta+1)t}.
\end{eqnarray*}
We also have from (\ref{A2}) that there exists \(C>0\) such that for all \(t\leq 1\) 
\begin{eqnarray*}
\mid \lambda\phi_\lambda^{\alpha,\beta}(t)\mid&\leq& C(1+\mid\lambda\mid){\rm e}^{-(\alpha+\beta+1)t}\\
&\leq&2C(1+\lambda^{2}){\rm e}^{-(\alpha+\beta+1)t}.
\end{eqnarray*}
Combining these inequalities, the lemma follows for \(n=0\).\\
We prove the case \(n\geq 1\) by induction with respect to \(n\) using the formula
\[
\dfrac{d^{n}}{dt^{n}}\phi^{\alpha,\beta}_{\lambda}(t)=\dfrac{-\Gamma(\alpha+1)}{4}((\alpha+\beta+1)^{2}+\lambda^{2})\sum_{k=0}^{n-1}2^{k}\dfrac{d^{k}}{dt^{k}}(\sinh(2t))\dfrac{d^{n-1-k}}{dt^{n-1-k}}\phi^{\alpha+1,\beta+1}_{\lambda}(t).
\] 
This last formula follows from the identity obtained by a simple calculation:
\begin{eqnarray*}
\dfrac{d}{dt}\phi^{\alpha,\beta}_{\lambda}(t)&=&\dfrac{-\Gamma(\alpha+1)}{4}((\alpha+\beta+1)^{2}+\lambda^{2})\sinh(2t)\phi^{\alpha+1,\beta+1}_{\lambda}(t).
\end{eqnarray*} 
\end{proof}
\begin{lemma} \label{L}Let \(f\) be a  \({\C}\)-valued function on \({\R}\) whose \(n\)th-order derivative is continuous. Assume  \(f(t_1)=f(t_2)=....=f(t_n)=0\), \(t_1,..,t_n\) being  real numbers such that  \(t_{1}>t_{2}>\ldots> t_{n}\). Then for  any \(t\in {\R}\) we have
\begin{eqnarray}|\dfrac{f(t)}{\prod\limits_{i=1}^{n}(t-t_{i})}| &\leq& \sup_{s\in [\min(t,t_{n}),\max(t,t_{1})]}|f^{(n)}(s)|. \nonumber
\end{eqnarray}
\end{lemma}
\begin{proof}
For \(n=1\) the result is obvious. We prove the case \(n\geq 1\) by induction on \(n\) using the fundamental theorem of calculus.\end{proof}
Recall that by the Cartan decomposition \(G=K\exp\mathfrak{p}\) each \(g\in G\) can be  uniquely written \(\displaystyle g=\pi_0(g)\exp X(g)\) where \(\pi_0(g)\in K\) and \(X(g)\in \mathfrak{p}\).For \(g=\kappa_{1}(g){\rm e}^{A^+(g)})\kappa_{2}(g)\) we know that \(\pi_0(g)=\kappa_{1}(g)\kappa_{2}(g)\). 
\begin{lemma}\label{RE}
Let \(g,h\in SU(r,r+b)\) and \(H_T\in\mathfrak{a}^+\). Then
\begin{align}\label{D1}
\displaystyle\kappa_1(h^{-1}g)\kappa_2(h^{-1}g)=\kappa_1(h^{-1}\kappa_1(g){\rm e}^{A^+(g)})\kappa_2(h^{-1}\kappa_1(g){\rm e}^{A^+(g)})\kappa_2(g).
\end{align}
\begin{align}\label{D2}
\lim_{R\rightarrow +\infty}\tau_l(\kappa_{1}(g{\rm e}^{RH_T})\kappa_{2}(g{\rm e}^{RH_T}))=\tau_l(\kappa(g))
\end{align}
\end{lemma}
\begin{proof} 
The identity (\ref{D1}) is obvious.\\
We write \(g=\kappa(g)e^{H(g)}n(g)\) with respect to the Iwasawa decomposition \(G=KAN\). Then we have 
\[\pi_{0}(g{\rm e}^{RH_{T}})=\kappa(g)\pi_{0}(e^{H(g)}n(g){\rm e}^{RH_T})  \qquad \forall R>0.\]
Therefore it is enough to prove
\[\lim_{R\rightarrow +\infty}\tau_l(\pi_0({\rm e}^{H(g)}n(g){\rm e}^{RH_T}))=1.\]
We first observe that if \(g=\begin{pmatrix}
A&B \\
 C&D\\ 
\end{pmatrix} \), then
\begin{align}\label{D3}
\tau_l(\pi_0(g))=(\frac{\det D}{\mid \det D\mid})^{-l}
\end{align}
We write \(e^{-RH_T}n(g)e^{RH_T}\) in \(r\times r+b\)-block notation as 
\[e^{-RH_T}n(g)e^{RH_T}=\begin{pmatrix}
 A_{1}(R)&B_{1}(R) \\
A_{2}(R)&B_{2}(R)\\
\end{pmatrix}.\]
Then a straightforward computation shows 
\begin{eqnarray*}
\tau_{l}(\pi_{0}(e^{H(g)+RH_T}e^{-RH_T}n(g)e^{RH_T}))=
 \left(\dfrac{det(\begin{pmatrix}\tanh(RT+S)\\
0_{b\times r}
\end{pmatrix}B_{1}(R)+B_{2}(R))}{|det(\begin{pmatrix}\tanh(RT+S)\\
0_{b\times r}
\end{pmatrix}B_{1}(R)+B_{2}(R)|}\right)^{-l}, \quad H(g)=H_S
\end{eqnarray*}
Since \(\lim_{R\rightarrow +\infty}{\rm e}^{-RH_T}n(g){\rm e}^{RH_T}=I_{2r+b}\) we get 
\[\lim_{R\rightarrow +\infty}\tau_{l}(\pi_{0}(e^{H(g)+RH_T})=1,\]
as to be shown.
\end{proof}

\end{document}